\documentclass[10pt]{article}
\usepackage{eurosym}
\usepackage{hyperref}
\usepackage{amssymb}
\usepackage{amsfonts}
\usepackage{graphicx}
\usepackage{amsmath}
\usepackage{float,color,ulem}
\usepackage{epsf,epsfig,subfigure}
\usepackage{float,color,ulem}
\usepackage{ulem}
\usepackage{bbm}
\usepackage{mathrsfs}
\usepackage[toc,page,title,titletoc,header]{appendix}

\allowdisplaybreaks

\newfloat{figure}{H}{lof}
\newfloat{table}{H}{lot}
\floatname{figure}{\figurename}
\floatname{table}{\tablename}
\newtheorem{theorem}{Theorem}[section]

\newtheorem{assumption}[theorem]{Assumption}

\newtheorem{corollary}[theorem]{Corollary}

\newtheorem{definition}[theorem]{Definition}
\newtheorem{example}[theorem]{Example}

\newtheorem{lemma}[theorem]{Lemma}

\newtheorem{remark}[theorem]{Remark}

\hypersetup{colorlinks=true, linkcolor=black, citecolor=black}
\newenvironment{proof}[1][Proof]{\noindent\textit{#1.} }{\hfill \rule{0.5em}{0.5em}}
\numberwithin{equation}{section}
\newcommand{\dint}{\displaystyle\int}

\begin{document}

\title{\textbf{A system of state-dependent delay differential equation
modelling forest growth I: semiflow properties}}
\author{ \textsc{Pierre Magal and Zhengyang Zhang } \\
%EndAName
{\small Univ. Bordeaux, IMB, UMR 5251, F-33076 Bordeaux, France}\\
{\small CNRS, IMB, UMR 5251, F-33400 Talence, France}}
\maketitle

\textbf{Abstract:} {\small In this article we investigate the semiflow properties of a class of state-dependent delay differential equations which is motivated by some models  describing the dynamics of the number of adult trees in forests. We investigate the existence and uniqueness of a semiflow in the space of Lipschitz and $C^1$ weighted functions. We obtain a blow-up result when the time approaches the maximal time of existence. We conclude the paper with an application of a spatially structured forest model. }\newline
\textbf{AMS Subject Classication} :  	34K05,  	37L99,  37N25.

\section{Introduction}

Let $\Omega$  be a compact subset of $\mathbb{R}^n$ (with  $n\geqslant 1$). Denote for simplicity that $C(\Omega):=C(\Omega ,\mathbb{R})$ and $C_{+}(\Omega):=C(\Omega ,[0,+\infty ))$. In this article we consider the following class of state-dependent delay differential equation: $\forall t\geqslant 0$ and $\forall x \in \Omega$,
\begin{equation}  \label{EQ1.1}
\left\{ 
\begin{array}{l}
\partial _{t}A(t,x)=F(A(t,.),\tau (t,.),A(t-\tau (t))(.,.))(x), \vspace{0.1cm} \\ 
\displaystyle\int_{-\tau (t,x)}^{0}f(A(t+s,.) )(x)ds=\displaystyle\int_{-\tau _{0}(x)}^{0}f(\varphi (s,.))(x)ds,
\end{array}
\right.
\end{equation}
where $A(t-\tau(t)) \in C(\Omega ^{2})$ is the map defined by 
\begin{equation} \label{EQ1.2}
A(t-\tau(t))(x,y) :=  A(t-\tau(t,x),y) 
\end{equation}
with the initial condition 
$$
A(t,x)=\varphi (t,x),\forall t \leqslant 0,
$$
and the initial distribution $\varphi$ belongs to  
\begin{equation*}
\begin{aligned}
\mathrm{Lip}_{\alpha}:=\{ \phi \in C((-\infty ,0],C(\Omega)): t \rightarrow e^{-\alpha \vert t \vert }\phi (t,.) \text{ is bounded and} \\
\text{Lipschitz continuous from }(-\infty ,0] \text{ to } C(\Omega)\}.
\end{aligned}
\end{equation*}
Recall that the space $\mathrm{Lip}_{\alpha}$ is a Banach space endowed with the norm 
\begin{equation*}
\Vert \phi\Vert _{\mathrm{Lip}_{\alpha}}:=\Vert \phi _{\alpha}\Vert _{\infty}+\Vert \phi _{\alpha}\Vert _{\mathrm{Lip}}
\end{equation*}
with 
\begin{equation*}
\left\Vert \phi _{\alpha}\right\Vert _{\infty}:=\sup_{t\leqslant 0} \Vert \phi _{\alpha}(t,. )\Vert _{\infty} \text{ and } \left\Vert \phi _{\alpha}\right\Vert _{\mathrm{Lip}}:=\sup_{t,s\leqslant 0:t \neq s}\frac{\Vert \phi _{\alpha}(t,. )- \phi _{\alpha}(s,. )\Vert _{\infty}}{\vert t -s \vert}  
\end{equation*}
and  $\phi_{\alpha}: (-\infty ,0] \to C(\Omega)$ is defined by  
\begin{equation} \label{EQ1.3}
\phi_{\alpha}(t,x):=e^{-\alpha \vert t\vert }\phi (t,x),\forall t\in (-\infty ,0],\forall x\in\Omega .
\end{equation}
In the rest of the paper the product space $\mathrm{Lip}_{\alpha}\times C(\Omega)$ will be endowed with the usual product norm 
\begin{equation*}
\left\Vert (\phi ,r) \right\Vert _{\mathrm{Lip}_{\alpha}\times C(\Omega)}:=\Vert \phi \Vert _{\mathrm{Lip}_{\alpha}}+\Vert r\Vert _{\infty },\forall \phi \in \mathrm{Lip}_{\alpha},\forall r\in C(\Omega).
\end{equation*}

We will make the following assumptions throughout this paper. 
\begin{assumption}\label{ASS1.1}
We assume that the map $F: C(\Omega)^{2}\times C(\Omega ^{2}) \rightarrow C(\Omega)$ is Lipschitz continuous on bounded sets, that is to say that for each constant $M>0$, there exists a constant $L(M) >0$ satisfying 
\begin{equation*}
\left\Vert F(u,v,w) -F(\widehat{u},\widehat{v},\widehat{w})\right\Vert _{\infty} \leqslant L(M) \left[ \Vert u-\widehat{u}\Vert _{\infty}+\Vert v-\widehat{v}\Vert _{\infty}+\Vert w-\widehat{w}\Vert _{\infty}\right]
\end{equation*}
whenever $\Vert u-\widehat{u}\Vert _{\infty} \leqslant M$, $\Vert v-\widehat{v}\Vert _{\infty} \leqslant M$ and $\Vert w-\widehat{w}\Vert _{\infty}\leqslant M$. 

We also assume that the map $f:C(\Omega) \to C(\Omega) $ is Lipschitz continuous and there exists a real number $M>0$ such that 
$$
0 < f(\varphi)(x) \leqslant M, \forall x \in \Omega \text{ and } \forall \varphi \in C( \Omega),
$$
and $f$ is monotone non-increasing, that is to say that 
$$
\varphi(x) \leqslant \widehat{\varphi}(x),\forall x \in \Omega \Rightarrow f(\varphi)(x) \geqslant f(\widehat{\varphi})(x), \forall x \in \Omega.
$$
\end{assumption}

Examples of state-dependent delay differential equations of this form has been considered  first by Smith \cite{Smith1, Smith2, Smith3, Smith4}. This idea has been suscessfully used in \cite{Brunner, Kloosterman} (see also the references therein). Our motivation to consider such a class of state-dependent delay differential equations comes from forest modelling. In \cite{Magal-Zhang1, Magal-Zhang2} such state-dependent delay differential equations have been used to model the competition for light between trees.   
\begin{example}[Finite number of species] The $m$-species case corresponds to the case $n=1$ and the domain $\Omega$ contains exactly $m$ elements. We can choose for example  
$$
\Omega=\{1,2,...,m\}
$$
and for $x=1,\ldots,m$,
$$
F(A(t,.),\tau (t,.),A(t-\tau (t))(.,.))(x)=G(x,A(t,.),\tau (t,x),A(t-\tau (t,x))(.))
$$
where $G:\Omega \times \mathbb{R}^3 \rightarrow \mathbb{R}$ is a map (see \cite{Magal-Zhang1} for more details). 
\end{example}

\begin{example}[Spatially structured case] For the spatially structured case, we can choose 
$$
\Omega=[0,x_{\max}]\times[0,y_{\max}].
$$
Moreover assume (for simplicity) that we have a single species, then we can choose
\begin{equation*}
\begin{split}
F(A(t,.),\tau (t,.),A_1(t,.))(x,y) := &  e^{-\mu_{J} \tau(t,x,y)}\frac{f(A(t,x,y))}{f(A_1(t,x,y))}(I-\varepsilon \Delta)^{-1}[\beta A_1(t,.)](x,y) \\
& -\mu _{A}A(t,x,y),
\end{split}
\end{equation*}
where $\Delta$ is the Laplacian operator on the domain $\Omega$ with periodic boundary conditions. This model corresponds to the spatially structured model in \cite{Magal-Zhang1}.
\end{example}

Let $A\in C((-\infty ,r],C(\Omega))$ (for some $r\geqslant 0$) be given. Then for each $t\leqslant r$, we will use the standard notation $A_{t}\in C((-\infty ,0],C(\Omega))$, which is the map defined by 
\begin{equation*}
A_{t}(\theta ,.)=A(t+\theta ,.),\forall \theta \leqslant 0.
\end{equation*}

For clarity we will specify the notion of a solution.
\begin{definition}\label{DE1.4}
Let $r\in (0,+\infty ]$. A solution of the system (\ref{EQ1.1}) on $[0,r)$ is a pair of continuous maps $A:(-\infty ,r)\rightarrow C(\Omega)$ and $\tau :[0,r)\rightarrow C_{+}(\Omega)$ satisfying 
\begin{equation*}
A(t,x)=\left\{ 
\begin{array}{l}
\varphi (0,x)+\displaystyle\int_{0}^{t}F(A(l,.),\tau(l,.),A(l-\tau (l))(.,.))(x)dl,\forall t\in [0,r),\vspace{0.1cm} \\ 
\varphi (t,x),\forall t\leqslant 0,
\end{array}
\right.
\end{equation*}
and 
\begin{equation*}
\int_{t-\tau (t,x)}^{t}f(A(s,.))(x)ds=\int_{-\tau _{0}(x)}^{0}f(\varphi (s,.))(x)ds,\forall t\in [0,r).
\end{equation*}
\end{definition}

In this problem the initial distribution is $(\varphi(t,x),\tau_{0}(x))$. The semiflow generated by (\ref{EQ1.1}) is 
\begin{equation*}
\mathcal{U}(t)(\varphi(.,x),\tau _{0}(x)) :=(A_{t}(.,x),\tau(t,x)),
\end{equation*}
where $A(t,x)$ and $\tau (t,x)$ is the solution of (\ref{EQ1.1}) with the initial distribution $(\varphi(t,x),\tau _{0}(x))$.

In order to clarify the notion of semiflow in this context, we introduce the following definition.

\begin{definition}
\label{DE1.5}Let $(M,d)$ be a metric space. Let $\mathcal{U}:D_{\mathcal{U}}\subset [0,+\infty)\times M\rightarrow M$ be a map defined on the domain 
\begin{equation*}
D_{\mathcal{U}}:=\left\{(t,x)\in [0,+\infty)\times M:0\leqslant
t<T_{BU}(x)\right\} ,
\end{equation*}
where $T_{BU}:M\rightarrow (0,+\infty]$ is a lower semi-continous map (the blow-up time). We will use the notation 
\begin{equation*}
\mathcal{U}(t)x:=\mathcal{U}(t,x),\forall (t,x)\in D_{\mathcal{U}}.
\end{equation*}
We say that $\mathcal{U}$ is \textbf{a maximal semiflow on} $M$ if the following properties are satisfied:

\begin{enumerate}
\item[(i)] $T_{BU}(\mathcal{U}(t)x)+t=T_{BU}(x)$, $\forall x\in M$, $\forall t\in [0,T_{BU}(x))$;

\item[(ii)] $\mathcal{U}(0)x=x$, $\forall x\in M$;

\item[(iii)] $\mathcal{U}(t)\mathcal{U}(s)x =\mathcal{U}(t+s)x$, $\forall t,s\in [0,T_{BU}(x))$ with $t+s<T_{BU}(x)$;

\item[(iv)] If $T_{BU}(x)<+\infty$, then 
\begin{equation*}
\lim_{t\nearrow T_{BU}(x)}d(\mathcal{U}(t)x,\mathbf{0}_{M})=+\infty .
\end{equation*}
\end{enumerate}

We will say that the semiflow $\mathcal{U}$ is \textbf{state variable continuous} if for each $t\geqslant 0$ the map $x\mapsto \mathcal{U}(t)x$ is continuous around each point where $\mathcal{U}(t)$ is defined. We will say that the semiflow $\mathcal{U}$ is \textbf{locally uniformly state variable continuous} if for each $r\geqslant 0$, 
\begin{equation}  \label{EQ1.4}
\lim_{x \rightarrow x_0} \sup_{t\in [0,r]}d(\mathcal{U}(t)x,\mathcal{U}(t)x_{0})=0
\end{equation}
whenever the map $A(t)$ is defined at $x$ and $x_{0}$.

We will say that the semiflow $\mathcal{U}$ is \textbf{continuous} if the map $(t,x)\mapsto \mathcal{U}(t)x$ is continuous from $D_{\mathcal{U}}$ into $M$.
\end{definition}

Actually the semiflow of the state-dependent delay differential equation (\ref{EQ1.1}) is not always continuous in time. Assume for example that $F\equiv 1$ a constant function. We fix $\alpha =0$ and $\Omega =\{1\}$. Consider $A(t)$ the solution of (\ref{EQ1.1}) with the initial distribution 
\begin{equation*}
(\varphi ,\tau _{0})=(0_{\mathrm{Lip}_{\alpha}},\tau _{0}).
\end{equation*}
Then the solution $A(t)$ is defined by 
\begin{equation*}
A(t)=\left\{ 
\begin{array}{l}
t,\text{ if }t\geqslant 0, \\ 
0,\text{ if }t\leqslant 0.
\end{array}
\right.
\end{equation*}
Hence $t\mapsto A(t)$ is differentiable almost everywhere and 
\begin{equation*}
A^{\prime }(t)=\left\{ 
\begin{array}{l}
1,\text{ if }t>0, \\ 
0,\text{ if }t<0.
\end{array}
\right.
\end{equation*}
Therefore for each $\widehat{t}\geqslant 0$, 
\begin{equation*}
\lim_{t\rightarrow \widehat{t}}\left\Vert A_{t}-A_{\widehat{t}}\right\Vert _{\mathrm{Lip},(-\infty,0]}=\lim_{t\rightarrow \widehat{t}}\left\Vert A^{\prime }(t+.)-A^{\prime }(\widehat{t}+.)\right\Vert _{L^{\infty }\left(-\infty ,0\right)}=1.
\end{equation*}
Therefore due to the possible discontinuity of $A^{\prime }(t)$ at time $t=0$, the semiflow is not continuous in time.

The following theorem is the main result of this section.

\begin{theorem}\label{TH1.6} 
There exists a maximal semiflow $\mathcal{U}:D_{\mathcal{U}}\subset [0,+\infty )\times \mathrm{Lip}_{\alpha}\times C_{+}(\Omega) \rightarrow \mathrm{Lip}_{\alpha}\times C_{+}(\Omega)$ and its corresponding blow-up time $T_{BU}:\mathrm{Lip}_{\alpha}\times C_{+}(\Omega) \rightarrow (0,+\infty ]$ such that for each initial distribution $(\varphi ,\tau _{0}) \in \mathrm{Lip}_{\alpha}\times C_{+}(\Omega)$, there exists a unique solution $A:(-\infty ,T_{BU}(\varphi ,\tau _{0}))\rightarrow C_{+}(\Omega)$ and $\tau :[0,T_{BU}(\varphi ,\tau _{0}))\rightarrow C_{+}(\Omega)$ of (\ref{EQ1.1}) satisfying 
\begin{equation*}
\mathcal{U}(t)(\varphi ,\tau _{0})(x) =(A_{t}(.,x),\tau (t,x)),\forall t\in [0,T_{BU}(\varphi ,\tau_{0})), \forall x\in\Omega .
\end{equation*}
Moreover if $T_{BU}(\varphi ,\tau_{0})<+\infty$, then 
\begin{equation*}
\limsup\limits_{t\nearrow T_{BU}(W_{0})}\Vert A(t,.)\Vert _{\infty}=+\infty .
\end{equation*}
Furthermore the semiflow $\mathcal{U}$  has the following properties:

\begin{enumerate}
\item[(i)] The map $T_{BU}$ is lower semi-continuous and $D_{\mathcal{U}}$ is relatively open in $[0,+\infty) \times \mathrm{Lip}_{\alpha}\times C_{+}(\Omega)$.

\item[(ii)] The semiflow $\mathcal{U}$ is locally uniformly state variable continuous in $\mathrm{Lip}_{\alpha}\times C_{+}(\Omega)$.

\end{enumerate}
\end{theorem}

In the sequel we will use the notation
\begin{equation*}
BUC_{\alpha }:=\left\{ \phi \in C((-\infty ,0],C(\Omega)):\phi _{\alpha}\in BUC((-\infty ,0],C(\Omega))\right\}
\end{equation*}
where 
$$
\phi _{\alpha}(t,x):=e^{-\alpha \vert t\vert }\phi (t,x)
$$ 
and $BUC((-\infty ,0],C(\Omega))$ denotes the space of bounded uniformly continuous maps from $(-\infty ,0]$ to $C(\Omega)$. The space $BUC_{\alpha }$ is again a Banach space endowed with the norm 
\begin{equation*}
\left\Vert \phi \right\Vert _{BUC_{\alpha }}=\sup_{t\leqslant 0} \Vert \phi_{\alpha}(t,. )\Vert _{\infty}.
\end{equation*}
We will also use the notation 
\begin{equation*}
\begin{array}{r}
BUC_{\alpha }^{1}:=\left\{ \phi \in C^{1}((-\infty ,0],C(\Omega)):\phi _{\alpha}\in BUC((-\infty ,0],C(\Omega))\right. \\ 
\left. \text{and }\partial _{t}\phi _{\alpha}\in BUC((-\infty ,0],C(\Omega))\right\}
\end{array}
\end{equation*}
and the space $BUC_{\alpha}^{1}$ is again a Banach space endowed with the norm 
\begin{eqnarray*}
\Vert \phi \Vert _{BUC_{\alpha }^{1}} :=\Vert \phi _{\alpha}\Vert _{\infty }+\Vert \partial _{t}\phi _{\alpha}\Vert _{\infty} =\Vert \phi _{\alpha}\Vert _{\infty }+\Vert \phi _{\alpha}\Vert _{\mathrm{Lip}}.
\end{eqnarray*}
Now we consider the subset of $BUC_{\alpha }^{1}\times C_{+}(\Omega)$ 
\begin{equation*}
D_{\alpha }:=\left\{(\phi ,\tau _{0}) \in BUC_{\alpha}^{1}\times C_{+}(\Omega):\phi ^{\prime }(0,x)=F(\phi (0,.),\tau_{0}(.),\phi (-\tau_{0}(.),.))(x),\forall x\in\Omega\right\}.
\end{equation*}
One can note that $D_{\alpha }$ is a closed subset of $BUC_{\alpha }^{1}\times C_{+}(\Omega)$. Therefore $D_{\alpha }$ is a complete metric space endowed with the distance 
\begin{equation*}
d_{D_{\alpha }}\left(( \phi ,\tau _{0}) ,( \widehat{\phi},\widehat{\tau}_{0}) \right) :=\Vert \phi -\widehat{\phi}\Vert _{BUC_{\alpha }^{1}}+\Vert \tau _{0}-\widehat{\tau}_{0}\Vert _{\infty }.
\end{equation*}
We also have 
\begin{equation*}
D_{\alpha }\subset BUC_{\alpha }^{1}\times C_{+}(\Omega)\subset \mathrm{Lip}_{\alpha}\times C_{+}(\Omega),
\end{equation*}
and the topology of $BUC_{\alpha }^{1}\times C_{+}(\Omega)$ and $\mathrm{Lip}_{\alpha}\times C_{+}(\Omega)$ coincide on $D_{\alpha }$. We have the following results.
\begin{theorem}\label{TH1.7} 
When $\tau_{0}$ is fixed, the subdomain $D_{\alpha}$ is dense in $BUC_{\alpha }\times C_{+}(\Omega)$, namely
\begin{equation*}
\overline{D}_{\alpha}^{BUC_{\alpha }\times C_{+}(\Omega)}=BUC_{\alpha }\times C_{+}(\Omega).
\end{equation*}
Moreover we have following properties:
\begin{enumerate}
\item[(i)] The subdomain $D_{\alpha }$ is positively invariant by the semiflow $\mathcal{U}$, that is to say that for each $(\varphi ,\tau_{0}) \in D_{\alpha }$, 
\begin{equation*}
\mathcal{U}(t)(\varphi ,\tau _{0}) \in D_{\alpha},\forall t\in [0,T_{BU}(\varphi ,\tau _{0})).
\end{equation*}
\item[(ii)] The semiflow $\mathcal{U}$ resticted to $D_{\alpha }$ is a continuous semiflow when $D_{\alpha }$ is endowed with the metric $d_{D_{\alpha }}$.
\end{enumerate}
\end{theorem}
Particularly, from (ii), we know that we can choose two different state space for $A_t$ ($\mathrm{Lip}_{\alpha}$ or $BUC_{\alpha}^{1}$), but only in the case of $BUC_{\alpha}^{1}$ can we get a continuous (in time) semiflow. 

In system (\ref{EQ1.1}), we can see from the second equation that the delay $\tau (t,x)$ is a solution of an integral equation. In the following (Lemma \ref{LE3.2}) we will see that the delay $\tau (t,x)$ can be seen as the solution of a partial differential equation, too. In Lemma \ref{LE3.6}, we will see that the delay $\tau(t,x)$ can be also regarded as a functional of $A_{t}$ and $(\varphi ,\tau _{0})$, which shows that it is actually a state-dependent delay. Specifically speaking, let $\delta_0 \in C_+(\Omega)$ be fixed, then we can define the map $\tau:D(\tau)\subset \mathrm{Lip}_{\alpha} \rightarrow [0 , + \infty)$ as the solution of 
\begin{equation} \label{EQ1.5}
\int_{-\tau(\phi ,x)}^{0} f(\phi (s,.))(x)ds=\delta_0(x)
\end{equation} 
with 
\begin{equation*}
D(\tau) =\left\{ \phi  \in \mathrm{Lip}_{\alpha} :\delta_0(x)<\int_{-\infty }^{0}f(\phi (s,.))(x)ds, \forall x \in \Omega \right\} .
\end{equation*}
Then we will see that 
$$
\tau(A_t,x)=\tau(t,x),\forall t\geqslant 0,
$$
and the first equation in \eqref{EQ1.1} can be rewritten as  
\begin{equation*}  
\partial _{t}A(t,x)=F(A(t,.),\tau (t,.),A(t-\tau (A_t,.),.))(x),\forall t\geqslant 0.
\end{equation*}
State-dependent delay differential equations have been used in the study of population dynamics of species \cite{Aiello, Alomari, Hbid, Kloosterman}. We refer in addition to \cite{Arino, Hartung} and the references therein for a nice survey on this topic. Moreover, the semiflow properties of a general class of state-dependent delay differential equations have been recently studied by Walther \cite{Walther} in $D_{\alpha }$. 

As an illustration, let us consider for example the map $F:BUC_{\alpha }^{1}\rightarrow\mathbb{R}$ in system (\ref{EQ1.1}) defined by
\begin{equation*}
F(\varphi):=\varphi(-\tau(\varphi))
\end{equation*}
where $\tau(\varphi)$ is defined as above in \eqref{EQ1.5}. Assume in addition that $f$ is continuously differentiable, then by Lemma \ref{LE3.4}, the state-dependent delay $\tau:BUC_{\alpha} \rightarrow C(\Omega)$ is $C^1$. Then for $\varphi _{0}\in BUC_{\alpha }^{1}$, we have
\begin{eqnarray*}
F(\psi +\varphi _{0})-F(\varphi _{0}) & = & (\psi +\varphi _{0})(-\tau (\psi +\varphi _{0}))-\varphi _{0}(-\tau (\varphi _{0})) \\
& = & \psi (-\tau (\psi +\varphi _{0}))+\varphi _{0}(-\tau (\psi +\varphi _{0}))-\varphi _{0}(-\tau (\varphi _{0})),
\end{eqnarray*}
from which we deduce the derivative
\begin{equation*}
DF(\varphi _{0})\psi =\psi(-\tau (\varphi _{0}))+\varphi _{0}^{\prime}(-\tau (\varphi _{0}))\cdot \partial_{\varphi} \tau (\varphi _{0})\psi ,
\end{equation*}
which satisfies the assumption (E) in Walther \cite{Walther}. 

In this article, we consider the pair $(A_t, \tau(t,.))$ as the state variable, and in this case we can also apply the result by Walther in \cite{Walther} to the delay differential equation 
\begin{equation}  \label{EQ1.6}
\left\{ 
\begin{array}{l}
\partial _{t}A(t,x)=F(A(t,.),\tau (t,.),A(t-\tau (t))(.,.))(x), \vspace{0.1cm} \\ 
\partial _{t} \tau (t,x)=1-\dfrac{f(A(t,.) )(x)}{f(A(t-\tau(t,.),.))(x)}. 
\end{array}
\right.
\end{equation}
Nevertheless the existence of a maximal semiflow as well as the blow-up time has been considered by Walther \cite{Walther}.

The article is organized as follows. In section 2 we prove that $D_{\alpha}$ is dense in $BUC_{\alpha }\times C_{+}(\Omega)$. In section 3 we prove some results regarding the delay $\tau(t,x)$. In sections 4 and 5 we will investigate the uniqueness and local existence of solutions, and the properties of semiflows. In the last section of the article, we will illustrate our results by proving the global existence of solutions for a spatially structured forest model.

\section{Density of the domain}

In this preliminary section we will prove the first result of Theorem \ref{TH1.7}, namely the density of $D_{\alpha}$ in the space $BUC_{\alpha }\times C_{+}(\Omega)$. 

\begin{proof}
Fix $\tau _{0}\in C_{+}(\Omega)$. Consider the space 
$$
\mathscr{X}:=C(\Omega)\times BUC_{\alpha}
$$
which is a Banach space endowed with the usual product norm. Define $\mathscr{A}:D(\mathscr{A})\subset\mathscr{X}\rightarrow\mathscr{X}$ a linear operator by
\begin{equation*}
\mathscr{A}\begin{pmatrix} 0_{C(\Omega)} \\ \varphi \end{pmatrix}:=\begin{pmatrix} -\partial _{t}\varphi(0,.) \\ \partial _{t}\varphi  \end{pmatrix},\forall \begin{pmatrix} 0_{C(\Omega)} \\ \varphi  \end{pmatrix} \in D(\mathscr{A}),
\end{equation*}
with 
$$
D(\mathscr{A}):=\{0_{C(\Omega)}\}\times BUC_{\alpha}^{1}.
$$ 
Then it is not difficult to prove that  
\begin{equation} \label{EQ2.1}
\overline{D(\mathscr{A})}=\{0_{C(\Omega)}\}\times BUC_{\alpha}.
\end{equation}

Moreover, the linear operator $\mathscr{A}$ is a Hille-Yosida operator (see \cite{Liu}). More precisely, we have $(0,\infty )\subset \rho (\mathscr{A})$ and for each $\lambda \in (0,\infty )$,
\begin{eqnarray*}
& & ( \lambda I-\mathscr{A}) ^{-1} 
\begin{pmatrix}
\alpha \\ 
\varphi
\end{pmatrix}
=\begin{pmatrix}
0_{C(\Omega)} \\ 
\psi
\end{pmatrix} \\ 
& \Leftrightarrow & \psi (\theta ,x)=\frac{1}{\lambda }e^{\lambda \theta }[\alpha +\varphi (0,x)] +\int _{\theta }^{0}e^{\lambda (\theta -l)}\varphi (l,x) dl.
\end{eqnarray*}
The linear operator $\mathscr{A}$ is Hille-Yosida since we have the following estimation 
\begin{equation} \label{EQ2.2}
\left\Vert (\lambda I-\mathscr{A}) ^{-n}\right\Vert _{\mathcal{L}(\mathscr{X}) }\leqslant \frac{1}{\lambda ^{n}},\forall n\geqslant 1,\forall \lambda >0.
\end{equation}
By using \eqref{EQ2.1} and \eqref{EQ2.2} it follows that 
\begin{equation} \label{EQ2.3}
\lim_{\lambda \to +\infty} \left\Vert \lambda ( \lambda I-\mathscr{A}) ^{-1} \begin{pmatrix} 0_{C(\Omega)} \\ \psi \end{pmatrix}-\begin{pmatrix} 0_{C(\Omega)} \\ \psi \end{pmatrix} \right\Vert_{\mathscr{X}}=0, \forall \psi \in BUC_{\alpha}.
\end{equation}
We define the nonlinear map $\mathscr{F}:\overline{D(\mathscr{A})}\rightarrow\mathscr{X}$,
\begin{equation*}
\mathscr{F}\begin{pmatrix} 0_{C(\Omega)} \\ \varphi \end{pmatrix}:=\begin{pmatrix} F(\varphi(0,.),\tau _{0}(.),\varphi(-\tau _{0}(.),.)) \\ 0_{BUC_{\alpha}} \end{pmatrix},\forall  \varphi \in BUC_\alpha. 
\end{equation*}
We observe that 
\begin{eqnarray*}
& & (\varphi ,\tau _{0}) \in D_{\alpha} \\
& \Leftrightarrow & (\mathscr{A}+\mathscr{F})\begin{pmatrix} 0_{C(\Omega)} \\ \varphi \end{pmatrix}\in \overline{D(\mathscr{A})} \text{ with }\begin{pmatrix} 0_{C(\Omega)} \\ \varphi \end{pmatrix} \in D(\mathscr{A}) \\
& \Leftrightarrow & (I-\lambda\mathscr{A}-\lambda\mathscr{F})\begin{pmatrix} 0_{C(\Omega)} \\ \varphi \end{pmatrix}\in \overline{D(\mathscr{A})} \text{ with }\begin{pmatrix} 0_{C(\Omega)} \\ \varphi \end{pmatrix} \in D(\mathscr{A}),\forall \lambda >0.
\end{eqnarray*}
Let $\begin{pmatrix} 0_{C(\Omega)} \\ \psi \end{pmatrix}\in \{0_{C(\Omega)}\}\times BUC_{\alpha}$ be fixed. Then $\forall \lambda >0$, consider
\begin{equation*}
(I-\lambda\mathscr{A}-\lambda\mathscr{F})\begin{pmatrix} 0_{C(\Omega)} \\ \varphi _{\lambda}\end{pmatrix}=\begin{pmatrix} 0_{C(\Omega)} \\ \psi \end{pmatrix} \text{ with }\begin{pmatrix} 0_{C(\Omega)} \\ \varphi_\lambda \end{pmatrix} \in D(\mathscr{A}),
\end{equation*}
which is equivalent to the fixed point problem
\begin{equation*}
\begin{pmatrix} 0_{C(\Omega)} \\ \varphi _{\lambda}\end{pmatrix}=\lambda^{-1}\left(\lambda^{-1} I-\mathscr{A}\right)^{-1}\begin{pmatrix} 0_{C(\Omega)} \\ \psi \end{pmatrix}+\left(\lambda^{-1} I-\mathscr{A}\right)^{-1}\mathscr{F}\begin{pmatrix} 0_{C(\Omega)} \\ \varphi _{\lambda}\end{pmatrix}.
\end{equation*}
Define the map
\begin{equation*}
\Phi _{\lambda}\begin{pmatrix} 0_{C(\Omega)} \\ \varphi \end{pmatrix}:=\lambda^{-1} \left(\lambda^{-1} I-\mathscr{A}\right)^{-1}\begin{pmatrix} 0_{C(\Omega)} \\ \psi \end{pmatrix}+\left( \lambda^{-1} I-\mathscr{A}\right)^{-1}\mathscr{F}\begin{pmatrix} 0_{C(\Omega)} \\ \varphi \end{pmatrix}.
\end{equation*}
Then by using the fact that $\mathscr{F}$ is Lipschitz on bounded sets and $\mathscr{A}$ is a Hille-Yosida operator, one can prove that there exists $\eta=\eta(r)>0$ such that
\begin{equation*}
\Phi _{\lambda}(B_{\psi ,r})\subset B_{\psi ,r},\forall \lambda \in (0,\eta]
\end{equation*}
and $\Phi _{\lambda}$ is a strict contraction on $B_{\psi ,r}$, where 
$$
B_{\psi ,r}:=B\left(\begin{pmatrix} 0_{C(\Omega)} \\ \psi \end{pmatrix},r\right)
$$
is the ball with center $\begin{pmatrix} 0_{C(\Omega)} \\ \psi \end{pmatrix}$ and radius $r$ in $\overline{D(\mathscr{A})}=\{0_{C(\Omega)}\}\times BUC_{\alpha}$. Thus by the Banach fixed point theorem, $\forall \lambda\in (0,\eta]$, there exists $\begin{pmatrix} 0_{C(\Omega)} \\ \varphi _{\lambda} \end{pmatrix}\in B_{\psi ,r}$ satisfying 
\begin{equation*}
\Phi _{\lambda}\begin{pmatrix} 0_{C(\Omega)} \\ \varphi _{\lambda} \end{pmatrix}=\begin{pmatrix} 0_{C(\Omega)} \\ \varphi _{\lambda} \end{pmatrix}.
\end{equation*}
Finally, since $\begin{pmatrix} 0_{C(\Omega)} \\ \psi \end{pmatrix}\in \overline{D(\mathscr{A})}$ and by using \eqref{EQ2.2} and \eqref{EQ2.3}, we have
\begin{equation*}
\lim _{\lambda\rightarrow 0+} \left\Vert \begin{pmatrix} 0_{C(\Omega)} \\ \varphi _{\lambda} \end{pmatrix}-\begin{pmatrix} 0_{C(\Omega)} \\ \psi \end{pmatrix} \right\Vert_{\mathscr{X}}=0,
\end{equation*}
which completes the proof.
\end{proof}

\section{Properties of the integral equation for $\protect\tau(t,x)$}

In this section we will make the following assumption.

\begin{assumption}
\label{ASS3.1} Let $(\varphi ,\tau _{0}) \in C((-\infty ,0] ,C(\Omega)) \times C_+(\Omega)$. Let $A\in C((-\infty ,r) ,C(\Omega))$ (with $r\in (0,+\infty]$) be given and satisfy 
\begin{equation*}
A(t,.)=\varphi (t,.),\forall t\leqslant 0.
\end{equation*}
\end{assumption}

\begin{lemma}
\label{LE3.2} There exists a uniquely determined map $\tau:[0,r)\rightarrow C(\Omega)$ satisfying 
\begin{equation}  \label{EQ3.1}
\int_{t-\tau (t,x)}^{t}f(A(s,.))(x)ds=\int_{-\tau _{0}(x)}^{0}f(\varphi(s,.))(x)ds,\forall t\in [0,r),\forall x\in\Omega.
\end{equation}
Moreover this uniquely determined map $t\mapsto \tau (t,x)$ is continuously differentiable and satisfies the following equation 
\begin{equation}  \label{EQ3.2}
\left\lbrace
\begin{array}{l}
\partial _{t}\tau(t,x)=1-\dfrac{f(A(t,.))(x)}{f(A(t-\tau (t,x),.))(x)},\forall t\in [0,r),\forall x\in\Omega,\\
\tau (0,x)=\tau _{0}(x).
\end{array}
\right.
\end{equation}
Conversely if $t\mapsto \tau(t,x)$ is a $C^{1}$ map satisfying the above ordinary differential equation (\ref{EQ3.2}), then it also satisfies the above integral equation (\ref{EQ3.1}).
\end{lemma}

\begin{remark}
By using equation (\ref{EQ3.2}), it is easy to check that 
\begin{equation*}
\tau _{0}(x)>0\Rightarrow \tau (t,x)>0,\forall t\in [0,r),\forall x\in\Omega,
\end{equation*}
and 
\begin{equation*}
\tau _{0}(x)=0 \Rightarrow \tau (t,x)=0,\forall t\in [0,r),\forall x\in\Omega.
\end{equation*}
\end{remark}

\begin{proof}
\textit{Step 1 (Existence of $\tau(t,x)$):} By Assumption \ref{ASS1.1}, $f$ is strictly positive, so fix $t\in [0,r)$ and $x\in\Omega$, and by considering the function $\tau\mapsto\dint_{t-\tau }^{t}f(A(s,.))(x)ds$ and observing that 
\begin{equation*}
\int_{t-0}^{t}f(A(s,.))(x)ds=0\leqslant\int_{-\tau _{0}(x)}^{0}f(\varphi (s,.))(x)ds,
\end{equation*}
\begin{equation*}
\int_{t-(t+\tau _{0}(x))}^{t}f(A(s,.))(x)ds\geqslant \int_{-\tau _{0}(x)}^{0}f(\varphi (s,.))(x)ds,
\end{equation*}
it follows by the intermediate value theorem that there exists a unique $\tau (t,x)\in [0,t+\tau _{0}(x)]$ satisfying \eqref{EQ3.1}. 

\noindent\textit{Step 2 (The map $t\mapsto t-\tau(t,x)$ is increasing):}  First let's prove that the function $t\mapsto t-\tau(t,x)$ is increasing. Indeed, assume by contradiction that $t_{1}\leqslant t_{2}$ while $t_{1}-\tau (t_{1},x)>t_{2}-\tau (t_{2},x)$, namely we have
$$
t_{2}-\tau (t_{2},x)<t_{1}-\tau (t_{1},x)<t_{1}\leqslant t_{2}.
$$
Then by \eqref{EQ3.1} we have
\begin{eqnarray*}
& & \int _{t_{1}-\tau (t_{1},x)}^{t_{1}}f(A(s,.))(x)ds=\int _{t_{2}-\tau (t_{2},x)}^{t_{2}}f(A(s,.))(x)ds \\
& = & \int _{t_{2}-\tau (t_{2},x)}^{t_{1}-\tau (t_{1},x)}f(A(s,.))(x)ds+\int _{t_{1}-\tau (t_{1},x)}^{t_{1}}f(A(s,.))(x)ds+\int _{t_{1}}^{t_{2}}f(A(s,.))(x)ds,
\end{eqnarray*}
thus
$$
 \int _{t_{2}-\tau (t_{2},x)}^{t_{1}-\tau (t_{1},x)}f(A(s,.))(x)ds+\int _{t_{1}}^{t_{2}}f(A(s,.))(x)ds=0,
$$
which is impossible since the function $f$ is strictly positive. 

\noindent\textit{Step 3 (The continuity of the map $ x \mapsto \tau(t,x)$):} Next we will prove the continuity of the map $x\mapsto \tau (t,x)$. By step 2 we have $\forall t\in [0,r)$,
$$
t-\tau (t,x)\leqslant -\tau _{0}(x) \Leftrightarrow 0\leqslant \tau(t,x) \leqslant t-\tau_0(x).
$$
Now the boundness of the function $x \mapsto \tau(t,x)$ follows from the boundedness of $\tau_0(x)$. Then for any $t\in [0,r)$, we know that
$$
\tau ^{\infty}(t):=\sup _{x\in\Omega}\tau (x,t)<+\infty .
$$
Let $\xi (x):=\dint _{-\tau _{0}(x)}^{0}f(\varphi(s,.))(x)ds$, $\forall x\in\Omega$. Then for any $x_{0},x\in\Omega$, assume without loss of generality that $\tau _{0}(x_{0})>\tau _{0} (x)$, then
\begin{eqnarray*}
& & \vert\xi (x_{0})-\xi (x)\vert =\left\vert\int _{-\tau _{0}(x_{0})}^{0}f(\varphi(s,.))(x_{0})ds-\int _{-\tau _{0}(x)}^{0}f(\varphi(s,.))(x)ds\right\vert \\
& \leqslant & \left\vert\int _{-\tau _{0}(x_{0})}^{-\tau _{0}(x)}f(\varphi(s,.))(x_{0})ds\right\vert +\int _{-\tau _{0}(x)}^{0}\vert f(\varphi(s,.))(x_{0})-f(\varphi(s,.))(x)\vert ds \\
& \leqslant & \vert\tau _{0}(x_{0})-\tau _{0}(x)\vert f(m_{1})(x_{0})+\tau _{0}^{\infty}\Vert f\Vert _{\mathrm{Lip}}\sup _{s\in [-\tau _{0}^{\infty},0]}\vert\varphi (s,x_{0})-\varphi (s,x)\vert ,
\end{eqnarray*}
where 
$$
m_{1}:=\inf _{s\in [-\tau _{0}^{\infty},0]}\Vert\varphi (s,.)\Vert _{\infty}\text{ and }\tau _{0}^{\infty}:=\sup _{x\in\Omega}\tau _{0}(x).
$$
Then the continuity of $\xi(x)$ in $x$ follows from the continuity of $\tau _{0}(x)$ and $\varphi (t,x)$ in $x$. Now for fixed $t\in [0,r)$, by \eqref{EQ3.1} we have
\begin{eqnarray*}
& & \xi (x_{0})-\xi (x)=\int _{t-\tau (t,x_{0})}^{t}f(A(s,.))(x_{0})ds-\int _{t-\tau (t,x)}^{t}f(A(s,.))(x)ds \\
& = & \int _{t-\tau (t,x_{0})}^{t-\tau (t,x)}f(A(s,.))(x_{0})ds+\int _{t-\tau (t,x)}^{t}(f(A(s,.))(x_{0})-f(A(s,.))(x)),
\end{eqnarray*}
thus
\begin{equation*}
\int _{t-\tau (t,x_{0})}^{t-\tau (t,x)}f(A(s,.))(x_{0})ds\leqslant \vert\xi (x_{0})-\xi (x)\vert +\tau ^{\infty}(t)\Vert f\Vert _{\mathrm{Lip}}\sup _{s\in [-\tau ^{\infty}(t),0]}\vert A(s,x)-A(s,x_{0})\vert .
\end{equation*}
On the other hand, we have
\begin{equation*}
\int _{t-\tau (t,x_{0})}^{t-\tau (t,x)}f(A(s,.))(x_{0})ds \geqslant \vert\tau (t,x_{0})-\tau (t,x)\vert f(M_{1})(x_{0}),
\end{equation*}
where 
$$
M_{1}:=\sup _{s\in [-\tau _{0}^{\infty},t]}\Vert A(s,.)\Vert _{\infty}.
$$
Then there exists a constant $\eta :=\eta (t,x_{0})$ such that
\begin{equation*}
\vert\tau (t,x_{0})-\tau (t,x)\vert\leqslant \eta\left(\vert\xi (x_{0})-\xi (x)\vert +\sup _{s\in [-\tau ^{\infty}(t),0]}\vert A(s,x)-A(s,x_{0})\vert \right).
\end{equation*}
Then the continuity of $\tau (t,x)$ in $x$ follows from the continuity of $\xi(x)$ and $A(t,x)$ in $x$.

\noindent\textit{Step 4 (Differentiability of the map $t \mapsto \tau(t,x)$):} By applying the implicit function theorem to the map $\psi:[0,r]\times C(\Omega)\rightarrow C(\Omega)$ defined by 
$$
\psi(t,\gamma)(x)=\dint_{\gamma(x) }^{t}f(A(s,.))(x)ds
$$
(which is possible since $\dfrac{\partial\psi(t,\gamma)}{\partial\gamma}(\hat{\gamma})(x)=-f(A(\gamma(x),.))(x)\hat{\gamma}(x)$ and by Assumption \ref{ASS1.1}, $f$ is strictly positive), we deduce that $t\mapsto t-\tau (t,x)$ is continuously differentiable. By (\ref{EQ3.2}) we have for each $x\in\Omega$, $\forall t\geqslant s$,
\begin{equation*}
\tau(t,x)=\tau(s,x)+\int_{s}^{t}\left(1-\frac{f(A(l,.))(x)}{f(A(l-\tau (l,x),.))(x)}\right)dl.
\end{equation*}
Notice that the integrand in the above formula is continuous since the map $t\mapsto A(t-\tau(t,x),x)$ is continuous following from the fact that the maps $t\mapsto t-\tau(t,x)$ and $(t,x)\mapsto A(t,x)$ are continuous. Hence, $\tau\in C^{1}([0,r),C(\Omega))$.

\noindent\textit{Step 5 (\eqref{EQ3.2}$\Rightarrow$\eqref{EQ3.1}):} Conversely, assume that $\tau (t,x)$ is a solution of (\ref{EQ3.2}). Then 
\begin{equation*}
f(A(t,.))(x)=\left(1-\frac{\partial\tau(t,x)}{\partial t}\right) f(A(t-\tau (t,x),.))(x),\forall t\in [0,r),\forall x\in\Omega.
\end{equation*}
Integrating both sides with respect to $t$, we have 
\begin{equation*}
\int_{0}^{t}f(A(s,.))(x)ds=\int_{0}^{t}f(A(s-\tau (s,x),.))(x)\left(1-\frac{\partial\tau(s,x)}{\partial s}\right) ds.
\end{equation*}
Make the change of variable $l=s-\tau (s,x)$, we have $\forall t\in [0,r)$, $\forall x\in\Omega$,
\begin{equation*}
\begin{split}
& \int_{0}^{t}f(A(s,.))(x)ds=\int_{-\tau _{0}(x)}^{t-\tau (t,x)}f(A(l,.))(x)dl \\
\Leftrightarrow & \int_{t-\tau (t,x)}^{t}f(A(s,x))(x)ds+\int_{0}^{t-\tau(t,x)}f(A(s,.))(x)ds=\int_{-\tau _{0}(x)}^{t-\tau (t,x)}f(A(s,.))(x)ds \\
\Leftrightarrow & \int_{t-\tau (t,x)}^{t}f(A(s,.))(x)ds=\int_{-\tau _{0}(x)}^{t-\tau(t,x)}f(A(s,.))(x)ds-\int_{0}^{t-\tau (t,x)}f(A(s,.))(x)ds,
\end{split}%
\end{equation*}
this implies that $\tau(t,x)$ also satisfies the equation (\ref{EQ3.1}).
\end{proof}

In order to see that the delay $\tau $ is also a functional of $A_{t}$ and $( \varphi ,\tau _{0}) $, we define the following functional. We define the map $\widehat{\tau}:D(\widehat{\tau})\subset BUC_{\alpha } \times C(\Omega)\rightarrow C(\Omega)$ as the solution of 
\begin{equation}\label{EQ3.3}
\int _{-\widehat{\tau}(\phi ,\delta)(x)}^{0}f(\overline{\phi} (s,.))(x)ds=\delta(x)
\end{equation}
where 
\begin{equation*}
\overline{\phi}(s,x):=
\left\{ \begin{array}{l}
\phi (s,x), \text{ if } s \leqslant 0\\
\phi (0,x), \text{ if } s \geqslant 0.
\end{array}
\right.
\end{equation*}
Since by Assumption \ref{ASS1.1} the map $f(\phi (0,.))(x)>0$, then if $\delta(x)\leqslant 0$ we have $\widehat{\tau}(\phi ,\delta)(x) \leqslant 0$ and 
$$
\int _{-\widehat{\tau}(\phi ,\delta)(x)}^{0} f(\phi (0,.))(x)ds=\delta(x) \Leftrightarrow \widehat{\tau}(\phi ,\delta)(x)= \dfrac{\delta(x)}{f(\phi (0,.))(x)}. 
$$ 
It follows that the domain $D(\widehat{\tau})$ is given by 
\begin{equation*}
D(\widehat{\tau}) =\left\{( \phi ,\delta) \in BUC_{\alpha }\times C(\Omega):\delta(x)<\int _{-\infty }^{0}f(\phi (s,.))(x)ds\text{ if } \delta(x)> 0 \right\} .
\end{equation*}

For clarity we prove the following lemma. 
\begin{lemma} \label{LE3.4} 
For each $(\phi ,\delta) \in D(\widehat{\tau})$ there exists $\widehat{\tau}(\phi ,\delta) \in C(\Omega)$. 
\end{lemma}
\begin{proof} Let $(\phi ,\delta) \in D(\widehat{\tau})$ be fixed.
 
\noindent \textit{Step 1 (Existence of $\widehat{\tau}(\phi ,\delta)(x)$):} Let $x\in \Omega$ be fixed. If $\delta(x)\leqslant 0$, we have 
$$
\widehat{\tau}(\phi ,\delta)(x)= \dfrac{\delta(x)}{f(\phi (0,.))(x)}.
$$ 
If $\delta(x)> 0$, by the definition of the domain $D(\widehat{\tau})$ we have 
$$
\delta(x)<\int _{-\infty }^{0}f(\phi (s,.))(x)ds,
$$
therefore by the intermediate value theorem, we can find $ \widehat{\tau}(\phi ,\delta)(x) \in \mathbb{R} $ such that 
$$
\delta(x)=\int _{-\widehat{\tau}(\phi ,\delta)(x) }^{0}f(\phi (s,.))(x)ds.
$$
\noindent \textit{Step 2 (Boundedness of $\widehat{\tau}(\phi ,\delta)(x)$):} Assume by contradiction that $x \mapsto \widehat{\tau}(\phi ,\delta)(x)$ is unbounded. Since $\Omega$ is compact, we can find a converging sequence $x_n \to \overline{x} \in \Omega$ as $n \to + \infty$  such that 
$$
\lim_{n \to + \infty} \widehat{\tau}(\phi ,\delta)(x_n) = + \infty.
$$
It is sufficient to consider the case $ \delta(x_n) > 0$, since the case $\delta(x_n) < 0$ is explicit. By the continuity of the function $\delta$ we can assume that $\delta(\overline{x} )>0$. By the definition of the domain $D(\widehat{\tau})$ we have 
$$
\delta(\overline{x} ) < \int _{-\infty }^{0}f(\phi (s,.))(\overline{x})ds.
$$
So we can find a constant $M>0$ such that 
$$
\delta(\overline{x} ) < \int _{-M}^{0}f(\phi (s,.))(\overline{x})ds,
$$
and by continuity we can find a neighborhood $U$ of $\overline{x}$ such that 
$$
\delta(x ) < \int _{-M}^{0}f(\phi (s,.))(x)ds, \forall x \in U. 
$$
It follows that for all integer $n$ large enough,
$$
\widehat{\tau}(\phi ,\delta)(x_n) \leqslant M 
$$
a contradiction. \\
\noindent \textit{Step 3 (Continuity of the map $x \mapsto \widehat{\tau}(\phi ,\delta)(x)$):} From the previous part, we know that  
$$
\widehat{\tau}^{\infty}:=\sup_{x \in \Omega} \vert \widehat{\tau}(\phi ,\delta)(x) \vert < +\infty. 
$$
Fix $x_{0}\in\Omega$. If $\delta(x_{0}) <0$ there is nothing to prove. Let us assume that 
$$
\delta(x_{0}) \geqslant 0.
$$ 
Let $x_n \to x_{0}$ be a converging sequence. If $\delta(x_{0}) = 0$ and $\delta(x_{n})<0$, it is clear that $\widehat{\tau}(\phi ,\delta)(x_n) \to \widehat{\tau}(\phi ,\delta)(x_0)=0$. Without loss of generality we can assume that $\delta(x_{n})\geqslant 0$ for each integer $n\geqslant 0$.  By \eqref{EQ3.3} we have
\begin{eqnarray*}
\delta(x_{0})-\delta(x) & = & \int _{-\widehat{\tau}(\phi ,\delta)(x_{0})}^{0}f(\overline{\phi} (s,.))(x_{0})ds-\int _{-\widehat{\tau}(\phi ,\delta)(x)}^{0}f(\overline{\phi} (s,.))(x)ds \\
& = & \int _{-\widehat{\tau}(\phi ,\delta)(x_{0})}^{-\widehat{\tau}(\phi ,\delta)(x)}f(\overline{\phi} (s,.))(x_{0})ds+\int _{-\widehat{\tau}(\phi ,\delta)(x)}^{0}f(\overline{\phi} (s,.))(x_{0})ds \\
& & -\int _{-\widehat{\tau}(\phi ,\delta)(x)}^{0}f(\overline{\phi} (s,.))(x)ds.
\end{eqnarray*}
Assume without loss of generality that $\widehat{\tau}(\phi ,\delta)(x_{0})>\widehat{\tau}(\phi ,\delta)(x)$, then we have
\begin{eqnarray*}
& & \int _{-\widehat{\tau}(\phi ,\delta)(x_{0})}^{-\widehat{\tau}(\phi ,\delta)(x)}f(\overline{\phi} (s,.))(x_{0})ds \\
& = & (\delta(x_{0})-\delta(x))+\int _{-\widehat{\tau}(\phi ,\delta)(x)}^{0}(f(\overline{\phi} (s,.))(x)-f(\overline{\phi} (s,.))(x_{0}))ds \\
& \leqslant & \vert\delta(x_{0})-\delta(x)\vert +\int _{-\widehat{\tau}(\phi ,\delta)(x)}^{0}\Vert f\Vert _{\mathrm{Lip}}\left\vert \overline{\phi} (s,x)-\overline{\phi} (s,x_{0})\right\vert ds \\
& \leqslant & \vert\delta(x_{0})-\delta(x)\vert +\widehat{\tau}^{\infty}\Vert f\Vert _{\mathrm{Lip}}\sup _{s\in [-\widehat{\tau}^{\infty},0]}\vert \phi (s,x)-\phi (s,x_{0})\vert ,
\end{eqnarray*}
and
\begin{equation*}
\int _{-\widehat{\tau}(\phi ,\delta)(x_{0})}^{-\widehat{\tau}(\phi ,\delta)(x)}f(\overline{\phi} (s,.))(x_{0})ds \geqslant \vert \widehat{\tau}(\phi ,\delta)(x_{0})-\widehat{\tau}(\phi ,\delta)(x)\vert f(M_{1})(x_{0}),
\end{equation*}
where $M_{1}=\sup\limits _{s\in [-\widehat{\tau}^{\infty},0]}\Vert \phi (s,.)\Vert _{\infty}$. Thus there exists a constant $\eta:=\eta (x_{0})$ such that
\begin{equation*}
\vert \widehat{\tau}(\phi ,\delta)(x_{0})-\widehat{\tau}(\phi ,\delta)(x)\vert \leqslant \eta \left(\vert\delta(x_{0})-\delta(x)\vert +\sup _{s\in [-\widehat{\tau}^{\infty},0]}\vert \phi (s,x)-\phi (s,x_{0})\vert \right).
\end{equation*}
Then the result follows by the continuity of the functions $\delta$ and $\phi$ in $x$.
\end{proof}
\begin{lemma}\label{LE3.5}
Assume in addition that $f$ is continuously differentiable. Then the domain $D(\widehat{\tau})$ is an open subset of $BUC_{\alpha }\times C(\Omega)$ and the map $\widehat{\tau}:D(\widehat{\tau})\subset BUC_{\alpha } \times C(\Omega)\rightarrow C(\Omega)$ is continuously differentiable. 
\end{lemma}
\begin{proof}
Define the map $\Gamma:BUC_{\alpha }\times C(\Omega)\times C(\Omega)\rightarrow C(\Omega)$ by
\begin{equation*}
\Gamma(\phi,\delta,\gamma )(x):=\int_{-\gamma(x)}^{0}f(\overline{\phi}(s,.))(x)ds-\delta(x).
\end{equation*}
Since by Assumption \ref{ASS1.1} the map $f$ is $C^1$, so is the map $\Gamma$. By \eqref{EQ3.3}, we have $\Gamma (\phi,\delta,\widehat{\tau}(\phi,\delta))(x)=0$  and
\begin{equation*}
\partial_{\gamma} \Gamma(\phi,\delta,\gamma)(\hat{\gamma})(x)=f(\overline{\phi} (-\gamma(x),.) \hat{\gamma}(x).
\end{equation*}
Since $f$ is strictly positive, it follows that $\partial_{\gamma} \Gamma(\phi,\delta,\gamma)$ is invertible. The result follows by applying the implicit function theorem. 
\end{proof}
\begin{lemma}\label{LE3.6} 
Set 
\begin{equation*}
\delta_{0}(x):=\int_{-\tau _{0}(x)}^{0}f(\varphi (s,.))(x)ds,\forall x\in\Omega.
\end{equation*}
Then we have the following equality 
\begin{equation*}
\widehat{\tau}(A_{t},\delta_{0})(x)=\tau (t,x),\forall t\in (0,r),
\end{equation*}
where $\tau (t,x)$ is the solution of (\ref{EQ3.1}).
\end{lemma}
\begin{proof}
It is sufficient to observe that 
\begin{equation*}
\int_{-\widehat{\tau}(A_{t},\delta_{0})(x)}^{0}f(A_{t}(s,.))(x)ds=\int_{t-\widehat{\tau}(A_{t},\delta_{0})(x)}^{t}f(A(s,.))(x)ds=\delta_{0}(x).
\end{equation*}
\end{proof}

For simplicity, in case the function $\delta _{0}$ is not necessarily specified, we will write $\tau (\phi,x)$ instead of $\widehat{\tau}(\phi,\delta_{0})(x)$.

\begin{lemma}\label{LE3.7}
Let $\phi ,\widehat{\phi}\in BUC_{\alpha }$. Then 
\begin{equation*}
\phi \leqslant \widehat{\phi}\Rightarrow \tau (\phi ,x)\leqslant \tau (\widehat{\phi},x),\forall x\in\Omega.
\end{equation*}
\end{lemma}

\begin{proof}
Fix $x\in\Omega$, by (\ref{EQ3.3}) we have 
\begin{equation*}
\int_{-\tau (\phi ,x)}^{0}f(\phi(s,.))(x)ds=\delta _{0}(x)=\int_{-\tau (\widehat{\phi},x)}^{0}f(\widehat{\phi}(s,.))(x)ds.
\end{equation*}
Assume by contradiction that there exists $x\in\Omega$ such that $\tau (\phi ,x)>\tau (\widehat{\phi},x)$, then 
\begin{equation*}
\begin{split}
& \int_{-\tau (\phi ,x)}^{0}f(\phi (s,.))(x)ds-\int_{-\tau (\widehat{\phi },x)}^{0}f(\widehat{\phi}(s,.))(x)ds \\
= & \int_{-\tau (\phi,x)}^{-\tau (\widehat{\phi },x)}f(\phi(s,.))(x)ds+\int_{-\tau (\widehat{\phi},x)}^{0}[f(\phi(s,.))-f(\widehat{\phi}(s,.))](x)ds=0.
\end{split}%
\end{equation*}
As we have $f(\phi (s,.))(x)\geqslant f(\widehat{\phi}(s,.))(x)$, then 
\begin{equation*}
\int_{-\tau (\widehat{\phi},x)}^{0}[f(\phi (s,.))-f(\widehat{\phi}(s,.))](x)ds\geqslant 0.
\end{equation*}
Since by assumption $f(\phi(s,.))(x)>0$, we must have 
\begin{equation*}
\int_{-\tau (\phi ,x)}^{-\tau (\widehat{\phi},x)}f(\phi (s,.))(x)ds\leqslant 0,
\end{equation*}
which contradicts the fact that $-\tau (\phi ,x)<-\tau (\widehat{\phi},x)$. 
\end{proof}

%\begin{lemma}\label{LE3.8} 
%The map $t\mapsto t-\tau (t,x)=t-\tau (A_{t},x)$ is increasing with respect to $t$ from $\left[ 0,r\right) $ into $\mathbb{R}$.
%\end{lemma}
%
%\begin{proof}
%It is sufficient to observe that $t\mapsto \tau (t,x)$ satisfies (\ref{EQ3.2}), therefore  
%\begin{equation*}
%\frac{\partial}{\partial t}(t-\tau (t,x))=\frac{f(A(t,.))(x)}{f(A(t-\tau (t,x),.))(x)}>0, \forall t\in [0,r),
%\end{equation*}
%the result follows. 
%\end{proof}

In the following lemma we obtain some a priori estimates for the delay.
\begin{lemma}\label{LE3.8}
Assume that there exists a constant $M>0$ such that
\begin{equation}  \label{EQ3.4}
\sup_{t \in [0,r)}\Vert A(t,.)\Vert _{\infty } \leqslant M.
\end{equation}
Then 
\begin{equation*}
\tau _{\min }\leqslant \tau (A_{t},x)\leqslant \tau _{\max }, \forall t\in [0,r), \forall x\in\Omega, 
\end{equation*}
where the constants $\tau _{\min }$ and $\tau _{\max }$ are defined as follows:
\begin{equation*}
0\leqslant \tau _{\min }:=\frac{\displaystyle \inf_{x\in\Omega} [\tau _{0}(x)f(\varphi _{\mathrm{max}})(x)]}{\sup\limits _{x\in\Omega}f(-M_1)(x)} \leqslant \tau _{\max }:=\frac{\displaystyle \sup_{x\in\Omega} [\tau _{0}(x)f(-\varphi _{\mathrm{max}})(x)]}{\inf\limits _{x\in\Omega}f(M_{1})(x)} 
\end{equation*}
with
\begin{equation}\label{EQ3.5}
\begin{array}{l}
M_{1}:=\max\{ M,\varphi _{\mathrm{max}}\} \text{ and } \varphi _{\mathrm{max}}:=\sup\limits_{t\in [ -\tau _{0}^{\infty},0]}\Vert \varphi (t,.)\Vert _{\infty }
\end{array}
\end{equation}
and 
$$
\tau _{0}^{\infty}:=\sup\limits_{x\in\Omega}\tau _{0}(x).
$$ 
\end{lemma}
\begin{proof}
For any $x\in\Omega$ and $t\in [0,r)$, since by Assumption \ref{ASS1.1}, the map $f$ is decreasing, and by Lemma \ref{LE3.2} the map $t \mapsto t-\tau (A_{t},x)$ is increasing, therefore it follows that  
\begin{equation*}
\begin{split}
\int_{-\tau _{0}(x)}^{0}f(\varphi (s,.))(x)ds = & \int_{-\tau (A_{t},x)}^{0}f( A(t+s,.))(x)ds \\
= & \int_{t-\tau (A_{t},x)}^{t}f( A(l,.))(x)dl \\
\leqslant & \int_{t-\tau (A_{t},x)}^{t}f(-M_1)(x)ds =\tau (A_{t},x)\sup_{x\in\Omega}f(-M_1)(x).
\end{split}
\end{equation*}
Then $\forall x\in\Omega$, $\forall t\in [0,r)$,
\begin{equation*}
\tau (A_{t},x) \geqslant \frac{\displaystyle\int_{-\tau _{0}(x)}^{0}f(\varphi(s,.))(x)ds}{\sup\limits _{x\in\Omega}f(-M_1)(x)} \geqslant \frac{\inf\limits _{x\in\Omega}[\tau _{0}(x)f(\varphi _{\mathrm{max}})(x)]}{\sup\limits _{x\in\Omega}f(-M_1)(x)}.
\end{equation*}
The derivation of the estimation from above for $\tau(t,x)$ is similar. 
\end{proof}
\begin{lemma}\label{LE3.9}
Let $M_{0}>0$ be fixed. Let $\varphi ,\tilde{\varphi}\in \mathrm{Lip}_{\alpha}$ and $\tau _{0},\tilde{\tau}_{0}\in C_{+}(\Omega)$ satisfy 
\begin{equation*}
\Vert \varphi \Vert _{\mathrm{Lip}_{\alpha}}+\Vert \tau _{0}\Vert _{\infty }\leqslant M_{0} \text{ and }\Vert \tilde{\varphi}\Vert _{\mathrm{Lip}_{\alpha}}+\Vert \tilde{\tau}_{0}\Vert _{\infty }\leqslant M_{0}.
\end{equation*}
Let $r \in (0,+\infty]$ be fixed. Let $A,\tilde{A}\in C((-\infty ,r),C(\Omega))$ be such that 
\begin{equation*}
A(t,x)=\varphi(t,x), \tilde{A}(t,x)=\tilde{\varphi}(t,x),\forall (t,x) \in (-\infty,0] \times \Omega.
\end{equation*}
Assume that 
\begin{equation*}
M:=\max \left\{ \sup_{t\in[0,r)}\left\Vert A(t,.)\right\Vert _{\infty},\sup_{t\in[0,r)}\Vert \tilde{A}(t,.)\Vert _{\infty }\right\} < +\infty.
\end{equation*}
Then there exists a constant $L_{\tau }>0$ such that $\forall t\in [0,r)$, $\forall x\in\Omega$, 
\begin{equation*}
\vert \widehat{\tau}(A_{t},\delta _{0})(x)-\widehat{\tau}(\tilde{A}_{t},\tilde{\delta}_{0})(x)\vert \leqslant L_{\tau } \left[\sup_{s \in [-\bar{\tau}_{0}^{\infty},r)}\Vert A(s,.)-\tilde{A}(s,.)\Vert _{\infty}+\Vert \delta _{0}-\tilde{\delta}_{0}\Vert _{\infty}\right],
\end{equation*}
where $\delta _{0}$ and $\tilde{\delta}_{0}$ are defined as in Lemma \ref{LE3.6} respectively with $(\varphi ,\tau _{0})$ and $(\tilde{\varphi},\tilde{\tau}_{0})$ and
\begin{equation*}
\bar{\tau}_{0}^{\infty}:=\max\left\{\sup _{x\in\Omega}\tau _{0}(x),\sup _{x\in\Omega}\tilde{\tau} _{0}(x)\right\}.
\end{equation*}
\end{lemma}

\begin{proof}
Let be $t\in [0,r)$ and $x\in\Omega$. Recall from Lemma \ref{LE3.6} that
\begin{equation*}
\delta _{0}(x)=\int _{-\tau _{0}(x)}^{0}f(\varphi (s,.))(x)ds \text{ and }\tilde{\delta}_{0}(x)=\int _{-\tilde{\tau}_{0}(x)}^{0}f(\tilde{\varphi}(s,.))(x)ds,\forall x\in\Omega.
\end{equation*}
Without loss of generality we may assume that $\widehat{\tau}(A_{t},\delta _{0})(x)>\widehat{\tau}(\tilde{A}_{t},\tilde{\delta}_{0})(x)>0$. Then we have 
\begin{eqnarray*}
& & \delta _{0}(x)-\tilde{\delta}_{0}(x)=\int _{-\tau _{0}(x)}^{0}f(\varphi (s,.))(x)ds -\int _{-\tilde{\tau}_{0}(x)}^{0}f(\tilde{\varphi}(s,.))(x)ds \\
&=& \int _{t-\widehat{\tau}(A_{t},\delta _{0})(x)}^{t}f(A(s,.))(x)ds-\int _{t-\widehat{\tau}(\tilde{A}_{t},\tilde{\delta}_{0})(x)}^{t}f(\tilde{A}(s,.))(x)ds \\
&=& \int _{t-\widehat{\tau}(A_{t},\delta _{0})(x)}^{t-\widehat{\tau}(\tilde{A}_{t},\tilde{\delta}_{0})(x)}f(A(s,.))(x)ds+\int _{t-\widehat{\tau}(\tilde{A}_{t},\tilde{\delta}_{0})(x)}^{t}\left[ f(A(s,.))(x)-f(\tilde{A}(s,.))(x)\right] ds.
\end{eqnarray*}
Since by Assumption \ref{ASS1.1}, $f$ is Lipschitz continuous, then
\begin{eqnarray*}
& & \int _{t-\widehat{\tau}(A_{t},\delta _{0})(x)}^{t-\widehat{\tau}(\tilde{A}_{t},\tilde{\delta}_{0})(x)}f(A(s,.))(x)ds \\
&=& \int _{t-\widehat{\tau}(\tilde{A}_{t},\tilde{\delta}_{0})(x)}^{t}\left[f(\tilde{A}(s,.))(x)-f(A(s,.))(x)\right] ds + (\delta _{0}(x)-\tilde{\delta}_{0}(x)) \\
& \leqslant & \Vert f\Vert _{\mathrm{Lip}}\int _{t-\widehat{\tau}(\tilde{A}_{t},\tilde{\delta}_{0})(x)}^{t}\Vert\tilde{A}(s,.)-A(s,.)\Vert _{\infty}ds+ \Vert\delta _{0}-\tilde{\delta}_{0}\Vert _{\infty} \\
& \leqslant & \bar{\tau}_{\max }\Vert f\Vert _{\mathrm{Lip}}\sup _{s\in [-\bar{\tau}_{0}^{\infty},r)}\Vert A(s,.)-\tilde{A}(s,.)\Vert _{\infty}+ \Vert\delta _{0}-\tilde{\delta}_{0}\Vert _{\infty}
\end{eqnarray*}
where $\bar{\tau}_{\max }:=\max\{\tau _{\max},\tilde{\tau}_{\max}\}$ and $\tau _{\max},\tilde{\tau}_{\max}$ are obtained in Lemma \ref{LE3.8}. On the other hand, we have 
\begin{equation*}
\int _{t-\widehat{\tau}(A_{t},\delta _{0})(x)}^{t-\widehat{\tau}(\tilde{A}_{t},\tilde{\delta}_{0})(x)}f(A(s,.))(x)ds \geqslant \left(\widehat{\tau}(A_{t},\delta _{0})(x)-\widehat{\tau}(\tilde{A}_{t},\tilde{\delta}_{0})(x)\right)\inf _{x\in\Omega}f(M_1)(x)
\end{equation*}
where $M_{1}:=\max \{\varphi _{\mathrm{max}}, \tilde{\varphi}_{\mathrm{max}},M \}$ and $\varphi _{\mathrm{max}},\tilde{\varphi}_{\mathrm{max}}$ are defined in \eqref{EQ3.5} respectively with $\varphi ,\tilde{\varphi}$. The result follows. 
\end{proof}

\section{Existence and uniqueness of solutions}
We start this section with two technical lemmas. 
\begin{lemma}\label{LE4.1}
Let $a<b$ be two real numbers. Let $\chi \in \mathrm{Lip}([a,b], C(\Omega))$. Then for each $ c\in (a,b)$ we have the following estimation 
\begin{equation*}
\Vert\chi\Vert _{\mathrm{Lip}([a,b],C(\Omega))}\leqslant \Vert\chi\Vert _{\mathrm{Lip}([a,c],C(\Omega))}+\Vert\chi\Vert _{\mathrm{Lip}([c,b],C(\Omega))}.
\end{equation*}
where 
$$
\Vert\chi\Vert _{\mathrm{Lip}(I,C(\Omega))}:=\sup_{t,s \in I:t \neq s} \frac{\Vert \chi(t,.)-\chi(s,.) \Vert_{\infty} }{\vert t-s \vert}. 
$$
\end{lemma}
\begin{proof}
Let $t,s\in [a,b]$ with $t>s$. Define the function $\rho:[0,1]\rightarrow\mathbb{R}$ by
\begin{equation*}
\rho(h):=\Vert\chi((t-s)h+s,.)-\chi(s,.)\Vert _{\infty}, \forall h \in [0,1].
\end{equation*}
Then we have $\forall h,\hat{h}\in [0,1]$,
\begin{eqnarray*}
\vert\rho(h)-\rho(\hat{h})\vert & \leqslant & \left\vert \Vert\chi((t-s)h+s,.)-\chi(s,.)\Vert _{\infty}-\Vert\chi((t-s)\hat{h}+s,.)-\chi(s,.)\Vert _{\infty}\right\vert \\
& \leqslant & \Vert\chi((t-s)h+s,.)-\chi((t-s)\hat{h}+s,.)\Vert _{\infty } \\
& \leqslant & \Vert\chi\Vert _{\mathrm{Lip}([a,b],C(\Omega))}\vert t-s\vert\vert h-\hat{h}\vert ,
\end{eqnarray*}
thus $\rho$ is Lipschitz continuous. Denote
\begin{equation*}
\mathrm{Lip}(\rho)(h):=\limsup _{\varepsilon\rightarrow 0^{+}}\frac{\rho(h+\varepsilon)-\rho(h)}{\varepsilon},
\end{equation*}
then
\begin{equation*}
\mathrm{Lip}(\rho)(h)\leqslant\left\{\begin{array}{l}
\Vert\chi\Vert _{\mathrm{Lip}([a,c],C(\Omega))}\vert t-s\vert ,\text{if }(t-s)h+s\in [a,c), \vspace{0.05cm}\\
\Vert\chi\Vert _{\mathrm{Lip}([c,b],C(\Omega))}\vert t-s\vert ,\text{if }(t-s)h+s\in [c,b].
\end{array}\right.
\end{equation*}
Since $\rho$ is Lipschitz continuous, by using Theorem 8.17 in page 158 of Rudin \cite{Rudin}, we deduce that $\rho$ is differentiable everywhere on a subset of the form $[0,1] \setminus N$ (where $N$ has null Lebesgue measure) and
$$
\rho(t)=\rho(0)+\int _{0}^{t}\rho ^{\prime}(l)dl, \forall t \in [0,1]. 
$$
By using the definition of $\mathrm{Lip}(\rho)(t)$ we deduce that 
$$
\rho ^{\prime}(t) \leqslant \mathrm{Lip}(\rho)(t) \leqslant C , \forall t \in [0,1] \setminus N,
$$
where $C:=\left[ \Vert\chi\Vert _{\mathrm{Lip}([a,c],C(\Omega))} +\Vert\chi\Vert _{\mathrm{Lip}([c,b],C(\Omega))} \right] \vert t-s\vert$. Therefore we obtain 
\begin{equation*}
\Vert \chi (t,.)-\chi(s,.)\Vert _{\infty}=\rho(1)-\rho(0)=\int _{0}^{1}\rho ^{\prime}(l)dl  \leqslant  \int _{0}^{1} C dl
\end{equation*}
which completes the proof.
\end{proof}

\begin{lemma}\label{LE4.2} 
Let $t\geqslant 0$. Assume that $A\in C((-\infty ,t],C(\Omega))$ and $A_{0}=\varphi$. Define for each $(\theta, x) \in (-\infty ,0] \times \Omega$,
$$
A_{t,\alpha}(\theta ,x):=e^{-\alpha\vert\theta\vert}A_{t}(\theta ,x) \text{ and } \varphi _{\alpha}(\theta ,x):=e^{-\alpha\vert\theta\vert}\varphi (\theta ,x). 
$$ 
Then we have the following estimations
\begin{equation}  \label{EQ4.1}
\Vert A_{t,\alpha}\Vert _{\infty }\leqslant \sup _{\theta\in [0,t]}\Vert e^{\alpha (\theta -t)}A(\theta ,.)\Vert _{\infty }+e^{-\alpha t}\Vert\varphi _{\alpha}\Vert _{\infty },
\end{equation}
\begin{equation}  \label{EQ4.2}
\Vert A_{t,\alpha}\Vert _{\mathrm{Lip}((-\infty ,0],C(\Omega))}\leqslant\Vert A_{t,\alpha}\Vert _{\mathrm{Lip}([-t,0],C(\Omega))}+e^{-\alpha t}\Vert\varphi _{\alpha}\Vert _{\mathrm{Lip((-\infty ,0],C(\Omega))}}
\end{equation}
and 
\begin{equation}  \label{EQ4.3}
\Vert A_{t}\Vert _{\mathrm{Lip}_{\alpha}}\leqslant \sup _{\theta\in [0,t]}\Vert e^{\alpha (\theta -t)}A(\theta ,.)\Vert _{\infty }+\Vert A_{t,\alpha}\Vert _{\mathrm{Lip}([-t,0],C(\Omega))}+e^{-\alpha t}\Vert \varphi \Vert _{\mathrm{Lip}_{\alpha}}.
\end{equation}
\end{lemma}
\begin{proof} 
We have for the supremum norm
\begin{eqnarray*}
 \Vert A_{t,\alpha}\Vert _{\infty } & =  &\sup _{\theta\leqslant 0}\Vert e^{\alpha \theta }A(t+\theta ,.)\Vert _{\infty } =e^{-\alpha t} \sup _{\theta\leqslant 0}\Vert e^{\alpha (t+\theta) }A(t+\theta ,.)\Vert _{\infty } \\
& \leqslant & \sup _{s \in [0,t]}\Vert e^{\alpha (s -t)}A(s ,.)\Vert _{\infty}+e^{-\alpha t}\Vert\varphi _{\alpha}\Vert _{\infty }.
\end{eqnarray*}
The result follows by using similar arguments combined with Lemma \ref{LE4.1} for the Lipschitz semi-norm.  
\end{proof}

\begin{lemma}[Uniqueness of solutions]\label{LE4.3} 
Let $\varphi \in \mathrm{Lip}_{\alpha}$ and $\tau _{0}\in C_{+}(\Omega)$ satisfy
\begin{equation*}
\left\Vert \varphi \right\Vert _{\mathrm{Lip}_{\alpha}}+\Vert \tau _{0}\Vert _{\infty }\leqslant M_{0}.
\end{equation*}
where $M_{0}>0$ is a given real number. Let $r\in(0,+\infty)$ be given. Then the equation (\ref{EQ1.1}) admits at most one solution $(A,\tau )\in C((-\infty ,r],C(\Omega))\times C([0,r],C(\Omega))$.
\end{lemma}
\begin{proof}
Suppose that there exist $(A^1,\tau^1 )$,$(A^2,\tau^2) \in C((-\infty ,r],C(\Omega))\times C([0,r],C(\Omega))$ two solutions of \eqref{EQ1.1} on $(-\infty ,r]$ with 
\begin{equation*}
(A^1_{0},\tau^1(0,.))=(A^2_{0},\tau^2(0,.))=(\varphi ,\tau _{0}).
\end{equation*}
Define 
\begin{equation*}
t_{0}=\sup \left\{ t\in [0,r]:A^1(s ,x)=A^2(s ,x),\tau^1(s ,x)=\tau^2(s,x),\forall s \in [0,t],\forall x\in\Omega\right\} .
\end{equation*}
Assume that $t_{0}<r$. We first observe that since $r$ is finite, we have 
\begin{equation*}
\tilde{K}_{0}:=\sup _{s \in [0,r]}\Vert A^1(s ,.)\Vert _{\infty }+\sup _{s \in [0,r]}\Vert A^2(s ,.)\Vert _{\infty}<+\infty .
\end{equation*}
By Lemma \ref{LE3.8}, $\tau^1(t,x)$ and $\tau^2(t,x)$ are also bounded from above (by $\tau^1_{\max}$ and $\tau^2_{\max}$ respectively) on $t\in [0,r]$. Since by Assumption \ref{ASS1.1}, $F: C(\Omega)^{2}\times C(\Omega ^{2}) \rightarrow C(\Omega)$ is Lipschitz on bounded sets and for each $i=1,2$, each $t \in [0,r]$ and each $x \in \Omega$,
$$
A^i(t,x)=\varphi(0,x)+\dint_{0}^{t}F(A^i(l,.),\tau^i(l,.),A^i(l-\tau^i(l))(.,.))(x)dl,
$$
it follows by using Lemma \ref{LE3.8} that for each $i=1,2$,
$$
\sup_{s \in [0,r]}\Vert F(A^i(s,.),\tau^i(s,.),A^i(s-\tau^i(s))(.,.))\Vert_{\infty} <+ \infty ,
$$
and
\begin{equation*}
K_{L}:=\Vert A^1 \Vert _{\mathrm{Lip}([0,r],C(\Omega))}+\Vert A^2 \Vert _{\mathrm{Lip}([0,r],C(\Omega))}<+\infty .
\end{equation*}
Set 
\begin{equation*}
K_{0}:=2(\tilde{K_{0}}+\varphi _{\mathrm{max}})+\tau^1_{\max}+\tau^2_{\max},
\end{equation*}
where $\varphi _{\mathrm{max}}$ is defined in (\ref{EQ3.5}). We have for each $t \in [t_{0},r]$ and each $ x\in\Omega$,
\begin{eqnarray*}
A^1(t,x)-A^2(t,x)& = & \dint_{0}^{t}\left[F(A^1(l,.),\tau^1(l,.),A^1(l-\tau^1(l))(.,.))(x)\right. \\
& & \left. -F(A^2(l,.),\tau^2(l,.),A^2(l-\tau^2(l))(.,.))(x)\right]dl ,
\end{eqnarray*}
thus by using the fact that $F$ is Lipschitz on bounded sets, we obtain 
\begin{eqnarray*}
\Vert A^1(t,.)-A^2(t,.)\Vert_{\infty} & \leqslant & (t-t_{0})L(K_{0}) \sup _{s\in [t_{0},t]}\left[ \Vert A^1(s,.)-A^2(s,.)\Vert _{\infty } \right. \\
& & + \Vert A^1(s-\tau^1 (s))(.,.)-A^2(s-\tau^2(s))(.,.)\Vert_{\infty } \\
& & + \left. \Vert \tau^1(s,.)-\tau^2(s,.)\Vert _{\infty } \right].
\end{eqnarray*}
Define 
$$
\Vert A^1_t-A^2_t\Vert _{\infty }:=\sup_{s \leqslant 0} \Vert A^1(t+s,.)-A^2(t+s,.) \Vert_{\infty}. 
$$
By Lemma \ref{LE3.9} we have for each $s \in [t_{0},t]$,
$$
\Vert \tau^1(s,.)-\tau^2(s,.)\Vert _{\infty }=\Vert\tau(A^1_{s},.)-\tau(A^2_{s},.)\Vert _{\infty }  \leqslant L_{\tau }\Vert A^1_s-A^2_s \Vert _{\infty }\leqslant L_{\tau }\Vert A^1_t-A^2_t \Vert _{\infty }.
$$
Thus 
\begin{eqnarray*}
& & \Vert A^1(s-\tau^1(s))-A^2(s-\tau^2(s))\Vert _{\infty } \\
& \leqslant & \Vert A^1(s-\tau^1(s))-A^1(s-\tau^2(s))\Vert _{\infty } +\Vert A^1(s-\tau^2(s)) -A^2(s-\tau^2(s))\Vert _{\infty }  \\
& \leqslant & K_L \Vert\tau^1 (s,.)-\tau^2(s,.)\Vert _{\infty }+ \Vert A^1_t-A^2_t\Vert _{\infty },
\end{eqnarray*}
hence 
$$
\Vert A^1(s-\tau^1(s))-A^2(s-\tau^2(s)) \Vert _{\infty }  \leqslant (K_L L_{\tau }+1) \Vert A^1_t-A^2_t\Vert _{\infty }.
$$
So we obtain for each $t \in [t_{0},r]$, 
$$
\Vert A^1_t-A^2_t\Vert _{\infty }  \leqslant  (t-t_{0})L(K_{0}) ((K_L+1 )L_{\tau }+2) \Vert A^1_t-A^2_t\Vert _{\infty }. 
$$
It follows that we can find $\varepsilon \in (0, r-t_0)$ such that 
$$
\Vert A^1_t-A^2_t\Vert _{\infty }=0, \forall t \in [t_{0},t_{0}+\varepsilon],
$$
which contradicts with the definition of $t_{0}$. Thus $t_{0}=r$.
\end{proof}

\begin{theorem}[Local existence of solutions]\label{TH4.4}
Let $M_{0}>0$ be fixed. Then for $M>M_{0}$, there exists a time $r=r(M_{0},M)>0 $ such that for each $(\varphi ,\tau _{0}) \in \mathrm{Lip}_{\alpha}\times C_{+}(\Omega)$ satisfying 
\begin{equation*}
\left\Vert \varphi \right\Vert _{\mathrm{Lip}_{\alpha}}+\Vert \tau _{0}\Vert _{\infty } \leqslant M_{0},
\end{equation*}
system (\ref{EQ1.1}) admits a unique solution $(A,\tau )\in C((-\infty ,r],C(\Omega))\times C([0,r],C(\Omega))$. Moreover,
\begin{equation*}
\Vert A(t,.)\Vert _{\infty}\leqslant M,\forall t\in [0,r].
\end{equation*}
\end{theorem}

\begin{proof}
\noindent \textit{Step 1 (Fixed point problem):} We start by defining the fixed point problem.  Let $M_{0}>0$ be fixed. Let $\varphi \in \mathrm{Lip}_{\alpha}$ and $\tau _{0}\in C_{+}(\Omega)$ satisfy 
\begin{equation*}
\left\Vert \varphi \right\Vert _{\mathrm{Lip}_{\alpha}}+\Vert \tau _{0}\Vert _{\infty} \leqslant M_{0}.
\end{equation*}
Let $M>M_{0}$ be fixed. Define 
\begin{equation*}
E_{\varphi }:=\{A\in C((-\infty ,r],C(\Omega)):A_{0}=\varphi \text{ and } \sup_{t\in [0,r]}\Vert A(t,.)\Vert _{\infty}\leqslant M \}
\end{equation*}
where $r$ will be determined later on. 

Let $\Phi :E_{\varphi}\rightarrow C((-\infty ,r],C(\Omega))$ be the map defined as follows: for each $t \in [0,r]$ and $x\in\Omega$,
\begin{equation}\label{EQ4.4}
\Phi (A)(t)(x):= \varphi (0,x)+\dint_{0}^{t}F(A(l,.),\tau(A_l,.),A(l-\tau(A_l,.),.))(x)dl 
\end{equation}
where $\tau(A_t,x)$ is the unique solution of the integral equation \eqref{EQ3.1}, and for $t\leqslant 0$ and $x\in\Omega$,
$$
\Phi (A)(t)(x):=\varphi (t,x).
$$

Set
\begin{equation}\label{EQ4.5}
\widetilde{M}:=\max\{ M,\varphi _{\mathrm{max}}\}
\end{equation}
where $\varphi _{\mathrm{max}}$ is defined in (\ref{EQ3.5}). For any $A\in E_{\varphi}$ and $t\in[0,r]$, we have 
\begin{eqnarray*}
\Vert \Phi (A)(t)\Vert_{\infty} & \leqslant & \Vert\varphi(0,.)\Vert_{\infty} +\int _{0}^{t}\Vert F(A(l,.),\tau(A_l,.),A(l-\tau(A_l,.),.)) \Vert_{\infty} dl\\
& \leqslant & M_{0}+\int _{0}^{t}\Vert F(A(l,.),\tau(A_l,.),A(l-\tau(A_l,.),.))-F(0,0,0)\Vert_{\infty} dl\\
& & +\int _{0}^{t}\Vert F(0,0,0)\Vert_{\infty} dl \\
&\leqslant & M_{0}+r[(2\widetilde{M}+\tau _{\max})L(2\widetilde{M}+\tau _{\max})+\Vert F(0,0,0)\Vert_{\infty} ].
\end{eqnarray*}
Since $M>M_{0}$ we can find $r_1>0$ such that for each $r \in (0,r_1]$, 
\begin{equation*}
M_{0}+r[(2\widetilde{M}+\tau _{\max})L(2\widetilde{M}+\tau _{\max})+\Vert F(0,0,0)\Vert_{\infty} ]\leqslant M,
\end{equation*}
and it follows that 
\begin{equation*}
\Phi (E_{\varphi})\subset E_{\varphi }.
\end{equation*}
\textit{Step 2 (Lipschitz estimation):} Set 
\begin{equation}\label{EQ4.6}
M_{L}:=\frac{M-M_0}{r} \geqslant (2\widetilde{M}+\tau _{\max})L(2\widetilde{M}+\tau _{\max})+\Vert F(0,0,0)\Vert _{\infty}.  
\end{equation}
For each $t,s \in [0,r]$ with $t \geqslant s$ and $A \in E_{\varphi}$ we have 
\begin{eqnarray*}
& & \frac{\Vert \Phi (A)(t)-\Phi (A)(s) \Vert_{\infty} }{\vert t-s\vert } \leqslant \frac{ \dint _{s}^{t}\Vert F(A(l,.),\tau(A_l,.),A(l-\tau(A_l,.),.))\Vert_{\infty}  dl}{\vert t-s\vert} \\
& \leqslant & \frac{ \dint _{s}^{t}\Vert F(A(l,.),\tau(A_l,.),A(l-\tau(A_l,.),.))-F(0,0,0)\Vert_{\infty} dl }{\vert t-s \vert} +\frac{ \dint _{s}^{t}\Vert F(0,0,0)\Vert_{\infty} dl }{\vert t-s\vert} ,
\end{eqnarray*}
thus 
$$
\Vert \Phi (A) \Vert_{\mathrm{Lip}([0,r],C(\Omega))} \leqslant M_{L}, \forall A \in E_{\varphi}.
$$
\noindent \textit{Step 3 (Iteration procedure):} Consider the sequence $\{ A^{n}\} _{n\in \mathbb{N}}\subset E_{\varphi}$ defined by iteration as follows: for each $(t,x)\in (-\infty,r] \times \Omega$,
\begin{equation*}
A^{0}(t,x)=\left\{ 
\begin{array}{l}
\varphi (0,x), \text{ if } t \in [0,r], \\
\varphi (t,x), \text{ if } t \leqslant 0,
\end{array}
\right.
\end{equation*}
and for each integer $n \geqslant 0$,
\begin{equation*}
A^{n+1}(t,x)=
\left\{ 
\begin{array}{l}
\Phi (A^{n})(t)(x), \text{ if } t \in [0,r], \\
\varphi (t,x), \text{ if } t \leqslant 0. 
\end{array}
\right.
\end{equation*}
From the step 2 and the definition of $A_0$, we know that for each integer $n \geqslant 0$,
$$
A^n \in \mathrm{Lip}([- \tau_0^{ \infty },r], C(\Omega)).
$$
and 
$$
\Vert A^n  \Vert_{\mathrm{Lip}([-\tau_0^{\infty },r],C(\Omega))}\leqslant \max\{M_L,\Vert \varphi\Vert _{\mathrm{Lip}([-\tau _{0}^{\infty},0],C(\Omega))}\} =:\widehat{M}_L. 
$$
For each integer $n,p\geqslant 0$, the maps $A^{n}$ and $A^{p}$ coincide for negative time $t$, therefore we can define
$$
\Vert A^{n}-A^{p}\Vert _{\infty}:=\sup_{t\leqslant r}\Vert A^{n}(t,.)-A^{p}(t,.) \Vert_{\infty}. 
$$
Next, we have $\forall n\in \mathbb{N}$, 
\begin{eqnarray*} 
& & \Vert A^{n+1}-A^{n}\Vert _{\infty} =\sup _{t\in [0,r]}\Vert\Phi (A^{n})(t)(.)-\Phi (A^{n-1})(t)(.)\Vert _{\infty} \\
& \leqslant & \dint _{0}^{r}\Vert F(A^{n}(l,.),\tau(A_l^n,.),A^{n}(l-\tau(A_l^n,.),.)) \\
& & -F(A^{n-1}(l,.),\tau(A_l^{n-1},.),A^{n-1}(l-\tau(A_l^{n-1},.),.)) \Vert_{\infty} dl \\
& \leqslant & rL(\widetilde{M}+\tau _{\max})\left[\Vert A^{n}-A^{n-1}\Vert _{\infty}
+\sup _{l\in[0,r]}\Vert \tau(A_l^{n},.)-\tau(A_l^{n-1},.)\Vert _{\infty} \right. \\
& & \left. +\sup _{l\in [0,r]}\Vert A^{n}(l-\tau(A_l^{n},.),.)-A^{n-1}(l-\tau(A_l^{n-1},.),.) \Vert_{\infty} \right], \\
& \leqslant &r L(\widetilde{M}+\tau _{\max}) \left[\Vert A^{n}-A^{n-1} \Vert _{\infty} +(1+\Vert A^{n}\Vert _{\mathrm{Lip}([-\tau _{0}^{\infty},r],C(\Omega))})\right.\cdot \\ 
& & \left.\sup _{l\in [0,r]}\Vert \tau(A_l^{n},.)-\tau(A_l^{n-1},.)\Vert _{\infty} \right]. 
\end{eqnarray*}
By Lemma \ref{LE3.9} we obtain for each integer $n \geqslant 1$, 
\begin{eqnarray*} 
\left\Vert A^{n+1}-A^{n}\right\Vert _{\infty}  & \leqslant & r C \Vert A^{n}-A^{n-1} \Vert _{\infty}
\end{eqnarray*}
with $C:= \left(2+L_{\tau }(1+\widehat{M}_{L}) \right)L(\widetilde{M}+\tau _{\max}) $. 

Now we can find $r_2 \in (0,r_1]$ such that $r_2C <1/2$, and for each $r \in (0,r_2]$ we have 
$$
\left\Vert A^{n+1}-A^{n}\right\Vert _{\infty} \leqslant \frac{1}{2^n} \Vert A^{1}-A^{0}\Vert _{\infty}, \forall n \geqslant 1. 
$$
It follows that $\{A^n \vert_{[0,r]}\}$ is a Cauchy sequence in the space $C([0,r],C(\Omega))$ and coincides with $\varphi$ for negative $t$.  Define  
$$
A(t,x):=\left\{ 
\begin{array}{l}
\displaystyle \lim_{n \to +\infty} A^n(t,x), \text{ if } t \in [0,r], x\in\Omega , \\
\varphi(t,x), \text{ if } t \leqslant 0,x\in\Omega .
\end{array}
\right.
$$
Then we have 
$$
\lim_{n \to + \infty} \left\Vert A^{n}-A\right\Vert _{\infty} =0. 
$$
By using again Lemma \ref{LE3.9}, we have 
\begin{eqnarray*}
\sup _{l\in [0,r]}\Vert \tau(A_l^{n},.)-\tau(A_l,.)\Vert _{\infty} \leqslant L_{\tau } \left\Vert A^{n}-A\right\Vert _{\infty}, \forall n \geqslant 1. 
\end{eqnarray*}
Finally by taking the limit on both sides of the equation 
\begin{equation*}
A^{n+1}(t,x):= \varphi (0,x)+\dint_{0}^{t}F(A^n(l,.),\tau(A_l^n,.),A^n(l- \tau(A_l^n,.),.))(x)dl,
\end{equation*}
we deduce that the couple $(A, \tau(A_t,.) )$ is a solution of equation \eqref{EQ1.1}.  

\noindent \textit{Step 4: (Estimation of the solution)}
We observe that 
\begin{eqnarray*}
\Vert A^{n+1}-A^{0} \Vert _{\infty} & \leqslant & \Vert A^{n+1}-A^{n} \Vert _{\infty} + \Vert A^{n}-A^{n-1} \Vert _{\infty} + \ldots + \Vert A^{1}-A^{0} \Vert _{\infty} \\
& \leqslant & \sum _{k=0}^{n}\left( \dfrac{1}{2}\right)^{n}\Vert A^{1}-A^{0}\Vert _{\infty},
\end{eqnarray*}
therefore 
$$
\Vert A^{n+1}-A^{0} \Vert _{\infty}  \leqslant 2 \Vert A^{1}-A^{0}\Vert _{\infty},
$$
and by taking the limit when $n$ goes to infinity we obtain 
$$
\Vert A-A^{0} \Vert _{\infty} \leqslant 2 \Vert A^{1}-A^{0}\Vert _{\infty}.
$$
From the definition of $A^1$ and $A^0$ we have 
$$
\Vert A^{1}-A^{0}\Vert _{\infty} \leqslant \dint _{0}^{r}\Vert F(A^{0}(l,.),\tau(A_l^0,.),A^{0}(l-\tau(A_l^0,.),.)) \Vert_{\infty} dl,
$$
so 
$$
\Vert A^{1}-A^{0}\Vert _{\infty} \leqslant r C_1,
$$
where $C_1=\sup\limits _{l \in [0,r]}\Vert F(A^{0}(l,.),\tau(A_l^0,.),A^{0}(l-\tau(A_l^0,.),.)) \Vert_{\infty}$. It follows that for each $t \in [0,r]$,
$$
\Vert A(t,.) \Vert _{\infty} \leqslant \Vert A(t,.) -\varphi(0,.)\Vert _{\infty}+ \Vert \varphi(0,.)\Vert _{\infty}\leqslant \Vert A-A^{0} \Vert _{\infty}+ M_0,
$$
thus  
$$
\Vert A(t,.) \Vert _{\infty} \leqslant M_0+ 2 r C_1,
$$
and by choosing $r$ small enough we obtain $M_0+ 2 r C_1 \leqslant M$. The proof is completed. 
\end{proof}

From step 2 in the above proof combined with Lemma \ref{LE4.2}, we have the following corollary.
\begin{corollary}\label{CO4.5}
With the same notation as in Theorem \ref{TH4.4}, we have 
$$ 
A_t \in \mathrm{Lip}_\alpha, \forall t \in [0,r],
$$
and there exists $\widehat{M}:=\widehat{M}(M,\tau_{\max},\alpha)> M$ such that 
$$
\Vert A_t \Vert_{\mathrm{Lip}_\alpha}\leqslant \widehat{M}, \forall t \in [0,r]. 
$$
\end{corollary}

\section{Properties of the semiflow}

In this section we investigate the semiflow properties of the map $W:D(W) \subset [0,+\infty) \times \mathrm{Lip}_\alpha \times C(\Omega) \to  \mathrm{Lip}_\alpha \times C(\Omega) $ defined for each initial distribution $W_{0}=\begin{pmatrix}\varphi \\ \tau _{0}\end{pmatrix}\in \mathrm{Lip}_{\alpha}\times C_{+}(\Omega)$ as 
\begin{equation*}
W(t,W_{0})(x)=\begin{pmatrix}
A_{t}(.,x) \\
\tau(t,x))
\end{pmatrix}
\end{equation*}
where $\{(A_t,\tau(t,.))\}_{t \in [0, T_{BU}(W_{0}))}$ is the solution of \eqref{EQ1.1} with initial distribution $W_{0}$ which is defined up to the maximal time of existence  
\begin{equation*}
\begin{array}{r}
T_{BU}(W_{0})=\sup \{ t>0:\text{ there exists a solution of  \eqref{EQ1.1}  on the  interval } [0,t] \\ \text{ with the initial distribution } W_{0}\},
\end{array}
\end{equation*}
and the domain 
$$
D(W)= \bigcup_{W_0 \in \mathrm{Lip}_{\alpha}\times C_{+}(\Omega) } [0,T_{BU}(W_{0}) )\times \{ W_0 \}. 
$$
We first observe by Theorem \ref{TH4.4} that we must have $T_{BU}(W_{0})>0$.

\begin{proof} \textit{(First part of Theorem \ref{TH1.6})} In this proof, we will verify that $\mathcal{U}$ satisfies the properties (i)-(iv) of Definition \ref{DE1.5}. Suppose that $W_{0}\in \mathrm{Lip}_{\alpha}\times C_{+}(\Omega)$ satisfies $\Vert W_{0}\Vert _{\mathrm{Lip}_{\alpha}\times C(\Omega)}\leqslant M_{0}$, where $M_{0}$ is a positive constant, and that the semiflow $\mathcal{U}$ is defined by 
\begin{equation*}
\mathcal{U}(t)W_{0}(x):=\begin{pmatrix}
A_{t}(.,x) \\
\tau (t,x)
\end{pmatrix}
\end{equation*}
where the map $t \mapsto (A_t,\tau(t,.))$ belongs to $C( [0,T_{BU}(W_0)), \mathrm{Lip}_\alpha \times C(\Omega))$ is solution of \eqref{EQ1.1} with initial value $W_0$.  The property (ii) of Definition \ref{DE1.5} is trivially satisfied, since by construction $\mathcal{U}(0)W_{0}=W_{0}$. 

\noindent \textit{Step 1 ((i) and (iii) of Definition \ref{DE1.5}):} By Lemma \ref{LE3.2}, the map $t \mapsto (A_t,\tau(t,.))$ is the solution of \eqref{EQ1.1} if and only if for each $t \in [0,T_{BU}(W_0))$ and each $x \in \Omega$ the equations 
\begin{equation}\label{EQ5.1}
\left\{\begin{array}{l}
A(t,x)=\varphi (0,x)+\dint _{0}^{t} F(A(l,.),\tau(l,.),A(l-\tau (l))(.,.))(x)dl, \vspace{0.1cm} \\
\tau (t,x)=\tau _{0}(x)+\dint _{0}^{t} \left[1-\dfrac{f(A(l,.))(x)}{f(A(l-\tau (l))(.,.))(x)}\right]dl,
\end{array}\right.
\end{equation}
are satisfied together with the initial condition 
$$
(A_0, \tau(0,.))=W_{0}.
$$  
Let $s \in [0,T_{BU}(W_{0}))$. Let us prove that 
$$
s + T_{BU}(U(s)W_{0}) \leqslant T_{BU}(W_{0}) 
$$
and for each $t \in [0, T_{BU}(U(s)W_{0}))$,
\begin{equation}\label{EQ5.2}
\mathcal{U}(t+s)W_0=\mathcal{U}(t)\mathcal{U}(s)W_0. 
\end{equation}
For each $t \in [0, T_{BU}(U(s)W_{0}))$,
\begin{equation*}
\mathcal{U}(t)\mathcal{U}(s)(W_{0})=\mathcal{U}(t)\left(\begin{array}{c}
A_s \\
\tau (s,.)
\end{array}\right)=\left(\begin{array}{c}
\tilde{A}_{t} \\
\tilde{\tau}(t)
\end{array}\right)
\end{equation*}
where $t \mapsto (A_t,\tau (t,.))$ is the solution of \eqref{EQ1.1} with initial condition $W_0$ and $t \mapsto (\tilde{A}_t,\tilde{\tau} (t,.))$ is the solution of \eqref{EQ1.1} with initial condition $\mathcal{U}(s)(W_{0})$. Then we have for each $t \in [0, T_{BU}(U(s)W_{0}))$, 
\begin{equation}\label{EQ5.3}
\left\{\begin{array}{l}
\tilde{A}(t,x)=A(s,x)+\dint _{0}^{t} F(\tilde{A}(l,.),\tilde{\tau}(l,.),\tilde{A}(l-\tilde{\tau} (l))(.,.))(x)dl \vspace{0.1cm} \\
\tilde{\tau }(t,x)=\tau (s,x)+\dint _{0}^{t}\left[1-\dfrac{f(\tilde{A}(l,.))(x)}{f(\tilde{A}(l-\tilde{\tau} (l))(.,.))(x)}\right]dl
\end{array}\right.
\end{equation}
with initial condition 
$$
\tilde{A}_0=A_s,\tilde{\tau }(0,x)=\tau (s,x).
$$
Now by setting 
$$
(\bar{A}_t,\bar{\tau}(t,.) ):=
\left\{
\begin{array}{l}
(\tilde{A}_{t-s},\tilde{\tau}(t-s,.)), \text{ if }t \in [s,s+ T_{BU}(U(s)W_{0})), \\
(A_{t},\tau(t,.)), \text{ if } t \in [0,s],
\end{array}
\right.
$$
then by using \eqref{EQ5.1} and \eqref{EQ5.3} we deduce that $ t \mapsto (\bar{A}_t,\bar{\tau}(t,.) )$ is a solution of \eqref{EQ1.1} on the time interval $[0,s+ T_{BU}(U(s)W_{0}))$ with initial condition $W_0$. It follows that  
$$
s + T_{BU}(U(s)W_{0}) \leqslant T_{BU}(W_{0}). 
$$ 

Now assume that $t \mapsto (A_t,\tau(t,.))$ is a solution of \eqref{EQ1.1} on the time interval $[0,T_{BU}(W_{0}))$ with initial condition $W_0$. Let $s \in  [0,T_{BU}(W_{0}))$. Then by using \eqref{EQ5.1} it follows that $t \mapsto (A_{t+s},\tau(t+s,.))$ defined on $[0,T_{BU}(W_{0})-s)$ is a solution of \eqref{EQ1.1} with initial condition $U(s)W_0$. It follows that 
$$
T_{BU}(W_0)-s \leqslant T_{BU}(U(s)W_0)
$$
and the properties (i) and (iii) of Definition \ref{DE1.5} follow.  

\noindent \textit{Step 2 ((iv) of  Definition \ref{DE1.5}):} Assume that $T_{BU}(W_{0})<+\infty$. Suppose that $\Vert\mathcal{U}(t)W_{0}\Vert _{\mathrm{Lip}_{\alpha}\times C(\Omega)}$ does not go to $+\infty $ when $t\nearrow T_{BU}(W_{0})$. Then there exists a constant $M_{0}>0$ and a sequence $\{t_{n}\}\subset [0,T_{BU}(W_{0}))$ such that $\lim\limits_{n\rightarrow +\infty}t_{n}=T_{BU}(W_{0})$, and for any $n\in \mathbb{N}$, 
\begin{equation*}
\Vert\mathcal{U}(t_{n})W_{0}\Vert _{\mathrm{Lip}_{\alpha}\times C(\Omega)}=\left\Vert 
\begin{pmatrix}
A_{t_{n}} \\ 
\tau(t_n,.)
\end{pmatrix}
\right\Vert _{\mathrm{Lip}_{\alpha}\times C(\Omega)}\leqslant M_{0}.
\end{equation*}
Let $W_{0,n}:=\begin{pmatrix}
A_{t_{n}} \\ 
\tau(t_n,.)
\end{pmatrix}$. Then by the local existence Theorem \ref{TH4.4} and Corollary \ref{CO4.5} we can find a constant $\widehat{M} > M_{0}$ and a time $r=r(M_{0},\widehat{M})>0$ such that for any $n\in\mathbb{N}$,
$$
T_{BU}(W_{0,n}) \geqslant r. 
$$
Moreover, from the previous part of the proof we have 
\begin{equation*}
T_{BU}(W_{0})=t_n+T_{BU}(W_{0,n}) \geqslant t_{n}+r,
\end{equation*}
and when $n\rightarrow +\infty $ we obtain  
\begin{equation*}
T_{BU}(W_{0}) \geqslant T_{BU}(W_{0})+r
\end{equation*}
a contradiction since  $r>0$. Thus we have
\begin{equation}\label{EQ5.4}
\lim_{t\nearrow T_{BU}(W_{0})}\Vert \mathcal{U}(t)W_{0}\Vert _{\mathrm{Lip}_{\alpha}\times C(\Omega)}=+\infty .
\end{equation}
\end{proof}

\begin{proof} \textit{(Second part of Theorem \ref{TH1.6})} Let us prove that if $T_{BU}(W_{0})<+\infty$ then 
\begin{equation*}
\limsup\limits_{t\nearrow T_{BU}(W_{0})}\Vert A(t,.)\Vert _{\infty}=+\infty .
\end{equation*}
Assume that $T_{BU}(W_{0})<+\infty$ and assume by contradiction that $ \limsup\limits_{t\nearrow T_{BU}(W_{0})}\Vert A(t,.)\Vert _{\infty}<+\infty$.  Since the map $t\mapsto t-\tau (t,x)$ is increasing, we have
\[
-\tau _{0}(x)\leqslant t-\tau (t,x) \leqslant t <T_{BU}(W_{0}), \forall x\in\Omega ,
\]
therefore 
$$
0\leqslant \tau (t,x) \leqslant T_{BU}(W_{0})+\tau _{0}(x). 
$$
And by assumption $T_{BU}(W_{0})<+\infty$, then
$$
\limsup_{t \nearrow T_{BU}(W_{0})} \Vert \tau (t,.) \Vert_{\infty} <+\infty. 
$$ 
Moreover, for each $t \in [0, T_{BU}(W_{0}))$, 
$$
\partial_t A(t,x)=F(A(t,.),\tau(t,.),A(t-\tau (t))(.,.))(x),
$$
and since $F$ is Lipschitz on bounded sets and by Lemma \ref{LE4.2} we deduce that 
$$
\limsup_{t \nearrow T_{BU}(W_{0})}  \Vert e^{ \alpha . } A_t(.)\Vert_{\mathrm{Lip}} < + \infty, 
$$
which contradicts \eqref{EQ5.4}. 
\end{proof}

\begin{lemma}
\label{LE5.1} We have the following results:

\begin{enumerate}
\item[(i)] The map $W_{0} \mapsto T_{BU}(W_{0})$ is lower semi-continuous on $\mathrm{Lip}_{\alpha}\times C_{+}(\Omega)$.

\item[(ii)] For every $x\in\Omega$, for every $W_{0}\in \mathrm{Lip}_{\alpha}\times C_{+}(\Omega)$, $\widehat{T}\in (0,T_{BU}(W_{0}))$, and every sequence $\{W_{0}^{(n)}\}_{n\in \mathbb{N}}\subset \mathrm{Lip}_{\alpha}\times C_{+}(\Omega)$ satisfying 
\begin{equation*}
\lim_{n\rightarrow +\infty }W_{0}^{(n)}=W_{0} \text{ in } \mathrm{Lip}_{\alpha}\times C_{+}(\Omega),
\end{equation*}
we have 
\begin{equation}\label{EQ5.5}
\lim_{n\rightarrow +\infty }\sup _{t\in [0,\widehat{T}]}\left\Vert \mathcal{U}(t)W_{0}^{(n)}-\mathcal{U}(t)W_{0}\right\Vert _{\mathrm{Lip}_{\alpha}\times C(\Omega)}=0.
\end{equation}
\end{enumerate}
\end{lemma}

\begin{proof}
\textit{Step 1 (Fixed point problem):} Let $ t \mapsto (\bar{A}_t, \bar{\tau}(t,.))$ be a solution of system \eqref{EQ1.1} which exists up to the maximal time of existence $T_{BU}(\overline{W}_{0})$ with the initial distribution $\overline{W}_{0}=\begin{pmatrix}
\bar{\varphi} \\ 
\bar{\tau}_{0}
\end{pmatrix}
\in \mathrm{Lip}_{\alpha}\times C_{+}(\Omega)$. 

Let $t^{\ast}\in (0,T_{BU}(\overline{W}_{0}))$ be fixed. By construction the map $t \mapsto (\bar{A}_t, \bar{\tau}(t,.))$ is continuous from $[0,T_{BU}(\overline{W}_{0})) $  to $BUC_{\alpha}\times C_{+}(\Omega)$. Therefore 
$$
\sup_{t \in [0,t^{\ast}] } \left[\Vert \bar{A}_t \Vert_{BUC_{\alpha}} +  \Vert \bar{\tau}(t,.) \Vert_{\infty}\right] <+\infty,
$$ 
and since $\bar{A}(t,.)$ satisfies the equation \eqref{EQ1.1} for positive time $t$, it follows that 
$$
\widehat{M}:=\sup_{t \in [0,t^{\ast}] } \Vert \bar{A}_t \Vert_{\mathrm{Lip}_{\alpha}} <+\infty. 
$$
Let $t_{0}\in [0,t^{\ast}]$ and $r>0$ with $t_{0}+r<t^{\ast}$.  Let $\varepsilon >0$ be fixed. Let $W_{0}=\begin{pmatrix}
\varphi \\ 
\tau _{0}
\end{pmatrix}
\in \mathrm{Lip}_{\alpha}\times C_{+}(\Omega)$ satisfy
$$
\Vert \varphi -\bar{A}_{t_{0}}\Vert _{\mathrm{Lip}_{\alpha}}\leqslant \varepsilon \text{ and }\Vert \tau _{0}-\bar{\tau}(t_{0},.)\Vert _{\infty}\leqslant \varepsilon .
$$
Let $M>\varepsilon$ be fixed. Define the space
\begin{equation*}
E_{\varphi ,t_{0}}:=\{A\in BUC_{\alpha }((-\infty ,r],C(\Omega)): A_{0}=\varphi \text{ and }\sup_{t\in [0,r]}\Vert A(t,.)-\bar{A}_{t_{0}}(t,.)\Vert _{\infty}\leqslant M \} .
\end{equation*}
Let $\Phi :E_{\varphi ,t_{0}}\rightarrow C((-\infty ,r],C(\Omega))$ be the map defined by
\begin{equation} \label{EQ5.6}
\Phi (A)(t)(x):=\varphi(0,x)+\int _{0}^{t}F(A(l,.),\widehat{\tau}(A_{l},\delta _{0}),A(l-\widehat{\tau}(A_{l},\delta _{0}),.))(x)dl
\end{equation}
whenever $t \in [0,r] $ and $x\in\Omega$, and
$$
\Phi (A_{})(t)(x):=\bar{A}_{t_{0}}(t,x)
$$
whenever $t \leqslant 0$ and $x\in\Omega$. 

In the formula \eqref{EQ5.6} the delay $\widehat{\tau}(A_{t}, \delta _{0})(x)$ is the unique solution of  
$$
\int _{-\widehat{\tau}(A_{t},\delta _{0})(x)}^{0}f(A(t+s,.))(x)ds=\delta _{0}(x)
$$
where
$$
\delta _{0}(x):=\int _{-\tau_{0}(x)}^{0}f(\varphi(s,.))(x)ds. 
$$
For any $A\in E_{\varphi ,t_{0}}$, $t\in [0,r]$ and $x\in\Omega$, we have
\begin{eqnarray*}
& & \vert\Phi(A)(t)(x)-\bar{A}_{t_{0}}(t,x)\vert \\
& = & \left\vert \varphi(0,x)+\int _{0}^{t}F(A(l,.),\widehat{\tau}(A_{l},\delta _{0}),A(l-\widehat{\tau}(A_{l},\delta _{0}),.))(x)dl \right. \\
& & \left. - \bar{A}_{t_{0}}(0,x)-\int _{0}^{t} F(\bar{A}_{t_{0}}(l,.),\widehat{\tau}(\bar{A}_{t_{0}+l},\bar{\delta}_{0}),\bar{A}_{t_{0}}(l-\widehat{\tau}(\bar{A}_{t_{0}+l},\bar{\delta}_{0}),.))(x)dl \right\vert \\
& \leqslant & \varepsilon +\int _{0}^{t} \vert F(A(l,.),\widehat{\tau}(A_{l},\delta _{0}),A(l-\widehat{\tau}(A_{l},\delta _{0}),.))(x) \\
& & -F(\bar{A}_{t_{0}}(l,.),\widehat{\tau}(\bar{A}_{t_{0}+l},\bar{\delta}_{0}),\bar{A}_{t_{0}}(l-\widehat{\tau}(\bar{A}_{t_{0}+l},\bar{\delta}_{0}),.))(x) \vert dl,
\end{eqnarray*}
where
$$
\bar{\delta _{0}}(x)=\int _{-\bar{\tau}(t_{0},x)}^{0}f(\bar{A}(t_{0}+s,.))(x)ds=\int _{-\bar{\tau}_0(x)}^{0}f(\bar{\varphi}(s,.))(x)ds.
$$
By Lemma \ref{LE3.9}, we can find a constant $L_{\tau}>0$ such that 
$$
\Vert \widehat{\tau}(A_{l},\delta _{0})-\widehat{\tau}(\bar{A}_{t_{0}+l},\bar{\delta}_{0})\Vert _{\infty}
\leqslant L_{\tau}\left[\sup_{s\in [-\max\{\tau _{0}^{\infty},\bar{\tau}_{0}^{\infty}\},r]}\Vert A(s,.)-\bar{A}_{t_{0}}(s,.)\Vert _{\infty}+\Vert \delta _{0}-\bar{\delta}_{0}\Vert _{\infty}\right]
$$
where $\tau _{0}^{\infty}$ and $\bar{\tau}_{0}^{\infty}$ are defined as in Lemma \ref{LE3.8}.

By using the definition of $\delta _{0}$ and $\bar{\delta}_{0}$ we have
\begin{eqnarray*}
 \Vert\delta _{0}-\bar{\delta}_{0}\Vert _{\infty} & \leqslant & \left\Vert\int _{-\tau_{0}(x)}^{-\bar{\tau}(t_{0},x)}f(\varphi(s,.))(x)ds\right\Vert _{\infty} \\
&&+ \left\Vert\int _{-\bar{\tau}(t_{0},x)}^{0}[f(\varphi(s,.))(x)-f(\bar{A}(t_{0}+s,.))(x)]ds\right\Vert _{\infty} \\
& \leqslant & \Vert \tau_{0}(.)-\bar{\tau}(t_{0},.)\Vert _{\infty}\sup _{x\in\Omega}f(-\varphi_{\max})(x)\\
&&+\bar{\tau}_{0}^{\infty}\Vert f\Vert _{\mathrm{Lip}}\sup _{s\in [-\bar{\tau}_{0}^{\infty},0]}\Vert \varphi(s,.)-\bar{A}(t_{0}+s,.)\Vert _{\infty} \\
& \leqslant & \varepsilon \left(\sup\limits _{x\in\Omega}f(-\varphi _{\max})(x)+\bar{\tau}_{0}^{\infty}e^{\alpha\bar{\tau}_{0}^{\infty}}\Vert f\Vert _{\mathrm{Lip}}\right)
\end{eqnarray*}

where $\varphi _{\max}$ is defined as in Lemma \ref{LE3.8}.

From the above estimations, it follows that there exists a constant $M_{1}>0$ such that
$$
\Vert \widehat{\tau}(A_{l},\delta _{0})-\widehat{\tau}(\bar{A}_{t_{0}+l},\bar{\delta}_{0})\Vert _{\infty}\leqslant M_{1}.
$$
Now, similarly as in step 3 of the proof of Theorem \ref{TH4.4} to evaluate $\Vert A^{n+1}-A^{n}\Vert _{\infty}$, we deduce that there exists a constant $r_{1}>0$ such that for each $r\in (0,r_{1}]$, we have $\Phi(E_{\varphi ,t_{0}})\subset E_{\varphi ,t_{0}}$. \\
\textit{Step 2 (Lipschitz estimation):} Similarly as in step 2 of the proof of Theorem \ref{TH4.4}, we can deduce that there exists a constant $M_{L}>0$ such that
$$
\Vert \Phi (A)\Vert _{\mathrm{Lip}([0,r],C(\Omega))}\leqslant M_{L}, \forall A\in E_{\varphi ,t_{0}}.
$$
\textit{Step 3 (Iteration procedure):} Consider the sequence $\{A^{n}\}_{n\in\mathbb{N}}\subset E_{\varphi ,t_{0}}$ defined by iteration as follows: for each $(t,x)\in (-\infty ,r]\times\Omega$,
$$
A^{0}(t,x)= \bar{A}_{t_0}(t,x),
$$
and for each integer $n\geqslant 0$,
$$
A^{n+1}(t,x):=\left\{\begin{array}{l}
\Phi (A^{n})(t)(x),\text{ if }t\in [0,r], \\
\varphi (t,x),\text{ if }t\leqslant 0.
\end{array}\right.
$$
From step 2, we deduce that there exists a constant $\widehat{M}_{L}>0$ such that for each integer $n\geqslant 0$,
$$
\Vert A^{n}\Vert _{\mathrm{Lip}([-\tau _{0}^{\infty },r],C(\Omega))}\leqslant \widehat{M}_{L}.
$$
By suing the same argument as in step 3 of the proof of Theorem \ref{TH4.4}, we can find a constant $r_{2}\in (0,r_{1}]$ such that $\forall r\in (0,r_{2}]$,
\begin{equation*}
\Vert A^{n+1}-A^{n}\Vert _{\infty}\leqslant \frac{1}{2^{n}}\Vert A^{1}-A^{0}\Vert _{\infty},\forall n \geqslant 1.
\end{equation*}
It follows that $\{A^{n}\vert _{[0,r]}\}$ is a Cauchy sequence in the space $C([0,r],C(\Omega))$. Define
$$
A(t,x):=\left\{\begin{array}{l}
\lim\limits _{n\rightarrow +\infty}A^{n}(t,x),\text{ if }t\in [0,r], x\in\Omega , \\
\varphi (t,x),\text{ if }t\leqslant 0,x\in\Omega .
\end{array}\right.
$$
Then we have 
$$
\lim_{n \to + \infty} \left\Vert A^{n}-A\right\Vert _{\infty} =0,
$$
and we deduce that $(A,\widehat{\tau}(A_{l},\delta _{0}))$ is a solution of \eqref{EQ1.1} with the initial distribution $(\varphi ,\tau _{0})$. \\
\textit{Step 4 (Estimation of the solution):} As in step 4 of the proof of Theorem \ref{TH4.4}, we also have 
$$
\Vert A-A^{0}\Vert _{\infty}\leqslant 2\Vert A^{1}-A^{0}\Vert _{\infty}.
$$
Since we have
\begin{eqnarray*}
& & \Vert A^{1}-A^{0}\Vert _{\infty} \\
& = & \sup _{t\in [0,r]}\left\Vert \varphi (0,x)+\int _{0}^{t}F(A^{0}(l,.),\widehat{\tau}(A_{l}^{0},\delta _{0}),A^{0}(l-\widehat{\tau}(A_{l}^{0},\delta _{0}),.))(x)dl-\bar{A}_{t_{0}}(t,x)\right\Vert _{\infty} \\
& \leqslant & \sup _{t\in [0,r]}\Vert \varphi (0,.)-\bar{A}_{t_{0}}(0,.)\Vert _{\infty}\\
&& +\left\Vert\bar{A}_{t_{0}}(0,.)+\int _{0}^{t}F(A^{0}(l,.),\widehat{\tau}(A_{l}^{0},\delta _{0}),A^{0}(l-\widehat{\tau}(A_{l}^{0},\delta _{0}),.))(x)dl-\bar{A}_{t_{0}}(t,x)\right\Vert _{\infty}
\end{eqnarray*}
and since $\bar{A}(t,x)$ is a solution, the term in the above inequality is null, it follows that 
$$
\Vert A^{1}-A^{0}\Vert _{\infty} \leq \varepsilon . 
$$
Then by choosing $r \leq r_2$, we obtain 
$$
\Vert A(t,.)-\bar{A}_{t_{0}}(t,.)\Vert _{\infty}\leqslant 2 \varepsilon, \forall t \in [t_0, t_0+r]. 
$$
\textit{Step 5 (Convergence result):} Fix $r=\frac{t^\ast }{n} \leq r_2$ for some integer $n \geq 1$. Choose an initial value satisfying 
$$
\Vert \varphi -\bar{A}_{t_{0}}\Vert _{\mathrm{Lip}_{\alpha}}\leqslant \frac{\varepsilon}{2^{n+1}} \text{ and }\Vert \tau _{0}-\bar{\tau}(t_{0},.)\Vert _{\infty}\leqslant \frac{\varepsilon}{2^{n+1}} .
$$
By using the above prove result $t_0=0$, and  we decude that 
$$
\Vert A(t,.)-\bar{A}(t,.)\Vert _{\infty}\leqslant \frac{\varepsilon}{2^{n+1}}, \forall t \in [0,r]. 
$$
and by induction $t_0=kr$ for $k=0,...,n$ we obtain 
$$
\Vert A(t,.)-\bar{A}(t,.)\Vert _{\infty}\leqslant \frac{\varepsilon}{2^{n+1}}, \forall t \in [0,t^\ast],
$$
the result follows. 
\end{proof}

The part of of Theorem \ref{TH1.7}: time continuity of the semiflow in $BUC_{\alpha}^1$ is left to the reader.

\section{Application to the Forest model with space}

For $x \in \Omega := [0,1]$ and $t \geqslant 0$ we consider  
\begin{equation}\label{EQ6.1}
\partial_t A(t,x)= e^{-\mu_{J}\tau(t,x)}\dfrac{f(A(t,x))}{f(A(t-\tau(t,x),x))}B(t- \tau(t,x),x)-\mu_{A}A(t,x),
\end{equation} 
where the birth is defined by 
$$
B(t,x):=(I-\varepsilon \Delta)^{-1}[ \beta A(t,.)](x),
$$
where $\Delta$ is the Laplacian operator on the domain $\Omega$ with periodic boundary conditions, and the state-dependent delay satisfies  
$$
\displaystyle\int_{t-\tau(t,x)}^tf(A(\sigma,x))d\sigma=\displaystyle\int_{-\tau_0(x)}^0 f(\varphi (\sigma,x))d\sigma ,
$$
and $A(t,x)$ satisfies the initial condition
$$
A(t,x)=\varphi(t, x),  \forall t \leqslant 0,
$$
with 
$$
\varphi \in X_\alpha \text{ and } \varphi \geqslant 0. 
$$  
Then it is well known that $(I-\varepsilon \Delta)^{-1}$ is a positive operator, i.e. 
$$
(I-\varepsilon \Delta)^{-1} C_+(\Omega ) \subset C_+(\Omega ) 
$$
and 
$$
\Vert (I-\varepsilon \Delta)^{-1} \Vert_{\mathcal{L}(C(\Omega ))}=1.
$$

\subsection{Positivity}

Assume that 
$$
\varphi(t,x) \geqslant 0, \forall (t,x) \in (-\infty,0] \times \Omega.  
$$
Assume that the solution starting from this initial distribution exists up to the time $ T_{BU}>0$. Then we have for each $t \in [0, T_{BU})$,
\begin{equation*}
\begin{split}
A(t,x)= & e^{-\mu_A t} \varphi(0,x)+ \int_0^t  e^{-\mu_A (t-s)} e^{-\mu_{J}\tau(s,x)}\dfrac{f(A(s,x))}{f(A(s-\tau(s,x),x))}\cdot \\
& (I-\varepsilon \Delta)^{-1}[\beta A(s- \tau(s,x),.)](x) ds.
\end{split}
\end{equation*}
Since the operator $(I-\varepsilon \Delta)^{-1}$ preserves the positivity of the distribution and by Assumption \ref{ASS1.1}, $f$ is strictly positive, then the positivity of the solution follows by using fixed point arguments on $A(t,x)$ in the above integral equation. 

\subsection{Global existence}

Consider the number of juveniles
$$
J(t,x):=\int_{t-\tau(t,x)}^t e^{-\mu_J (t-s)} \beta (I-\varepsilon \Delta)^{-1}(A(s,.))(x)ds
$$
for each $t \in [0, T_{BU})$. It is clear that 
\begin{equation} \label{EQ6.2}
J(t,x) \geqslant 0, \forall t \in [0, T_{BU}).
\end{equation}
Moreover we have 
\begin{equation*}
\begin{split}
\partial_t J(t,x)= & \beta (I-\varepsilon \Delta)^{-1}(A(t,.))(x)- e^{-\mu_{J}\tau(t,x)}\dfrac{f(A(t,x))}{f(A(t-\tau(t,x),x))} \cdot \\
& \beta (I-\varepsilon \Delta)^{-1}(A(t- \tau(t,x),.))(x) -\mu_J J(t,x).
\end{split}
\end{equation*}
By summing equation \eqref{EQ6.1} and the above equation we obtain 
\begin{equation} \label{EQ6.3}
\partial_t [A(t,x)+J(t,x)]=\beta (I-\varepsilon \Delta)^{-1}(A(t,.))(x)-\mu_A A(t,x)-\mu_J J(t,x).
\end{equation}
Set
$$
u(t,x):=A(t,x)+J(t,x),
$$
then we have 
$$
\partial_t u(t,x) \leqslant \beta (I-\varepsilon \Delta)^{-1}(u(t,.))(x)-\mu u(t,x),
$$
where $\mu:= \min\{\mu_A, \mu_J\}$. By using comparison argument we deduce that 
$$
u(t,x) \leqslant e^{[\beta (I-\varepsilon \Delta)^{-1}-\mu]t} (u(0,.))(x), \forall t \in [0, T_{BU}).
$$
Therefore by using \eqref{EQ6.2}, we deduce that 
$$
A(t,x) \leqslant e^{[\beta (I-\varepsilon \Delta)^{-1}-\mu]t} (u(0,.))(x), \forall t \in [0, T_{BU}),
$$
and by using Theorem \ref{TH1.6}, we must have $T_{BU}=+\infty$.

\end{document}